\documentclass[11pt, reqno]{amsart}
\usepackage{amssymb,amsfonts,amsxtra,amsmath,enumerate,tikz-cd,xcolor,mathrsfs,dsfont,verbatim,centernot}
\usepackage[all]{xy}
\usepackage[cal=dutchcal,calscaled=.96]{mathalfa}

\usepackage{scalerel}
\usepackage{stackengine,wasysym}

\newcommand\reallywidetilde[1]{\ThisStyle{%
  \setbox0=\hbox{$\SavedStyle#1$}%
  \stackengine{-.1\LMpt}{$\SavedStyle#1$}{%
    \stretchto{\scaleto{\SavedStyle\mkern.2mu\AC}{.5150\wd0}}{.6\ht0}%
  }{O}{c}{F}{T}{S}%
}}

\newtheorem{theorem}{Theorem}[section]
\newtheorem{lemma}[theorem]{Lemma}
\newtheorem{corol}[theorem]{Corollary}
\newtheorem{proposition}[theorem]{Proposition}

\newtheorem{question}[theorem]{Question}

\newtheorem{definition}[theorem]{Definition}
\newtheorem{remark}[theorem]{Remark}

\newcommand{\Osr}{{\mathscr{O}}}
\newcommand{\Pbb}{{\mathbb{P}}}

\newcommand{\Cbb}{{\mathbb{C}}}
\newcommand{\Pbf}{{\mathbf{P}}}

\newcommand{\p}{\partial}

\DeclareMathOperator{\rank}{rank}
\DeclareMathOperator{\Ch}{\mathbf{Ch}}
\DeclareMathOperator{\supp}{supp}
\DeclareMathOperator{\pr}{pr}

\title[Lagrangian specialisation v.s. graph construction]{an approach to Lagrangian specialisation through MacPherson's graph construction}
\author{Xia Liao}
\address{
Department of Mathematical Sciences, 
Huaqiao University, 
Chenghua North Road 269, 
Quanzhou, Fujian, China
}
\email{xliao@hqu.edu.cn}

\begin{document}
\maketitle

\begin{abstract}
Let $f: M \to N$ be a holomorphic map between two complex manifolds. Assume $f$ is flat and sans \'{e}clatement en codimension 0 (no blowup in codimension 0). We study the theory of Lagrangian specialisation for such $f$, and prove a Gonz\'{a}lez-Sprinberg type formula for the local Euler obstruction relative to $f$. With the help of this formula and MacPherson's graph construction for the vector bundle map $f^*T^*N \to T^*M$, we find the Lagrangian cycle of the Milnor number constructible function $\mu$. As an application, we study the Chern class transformation of  $\mu$ when $f$ has finite contact type. 
\end{abstract}

\section{Introduction}

In \cite{MR804052}, Sabbah studied the 1-parameter specialisation of Lagrangian cycles. Let $Z$ be a closed analytic subset of $\mathbb{C}^n$, and let $f: Z \to S$ be a flat morphism into a 1-dimensional smooth base $S$. Sabbah considered the factorisation of $f$ through the graph

\begin{equation*}
\begin{tikzcd}
Z \arrow[rd,"f"] \arrow[r] & \mathbb{C}^n \times S \arrow[d] \\
                                  & S, 
\end{tikzcd}
\end{equation*}
hence he can view $Z$ as a family of subspaces of $\mathbb{C}^n$, and form a family of conormal spaces $T^{*}_{f^{-1}(s)}\mathbb{C}^n$. Picking an $s_0 \in S$, the Lagrangian specialisation at $s_0$ is the limit of $T^{*}_{f^{-1}(s)}\mathbb{C}^n$ as $s$ approaches $s_0$. This limit is still Lagrangian, but a priori has several irreducible components, each of which has some multiplicities. For $x \in f^{-1}(s_0)$, he considered a local Euler obstruction at $x$ relative to $f$, which we may denote by $\textup{Eu}_f(x)$, and related it in one way to the topological Euler characteristic of the Milnor fibre at $x$, and in another way to the local Euler obstruction of the Lagrangian specilisation at $x$.

One goal of the present paper is to extend the theory of Lagrangian specialisation to a situation where the dimension of the base can be greater than 1. More precisely, we begin with a holomorphic map $f: M \to N$ between two complex manifolds. We assume $f$ is flat and sans \'{e}clatement en codimension 0 (See \S \ref{noblowup}). Fixing a point $z\in M$ and factorising $f$ through the graph 

\begin{equation*}
\begin{tikzcd}
M \arrow[rd,"f"] \arrow[r] & M \times N \arrow[d] \\
                                  & N, 
\end{tikzcd}
\end{equation*}
we may consider the specialisation of $T^*_{f^{-1}(t)}M$ ($t \in N$) as $t$ approaches $f(z)$. As the 1-parameter deformation case, we will study the relative Euler obstruction at $z$, denoted by $\textup{Eu}_f(z)$, and show its restriction to the Milnor fibre $F_z$ gives $\chi(F_z)$, whereas its restriction to the special fibre $M_{f(z)}$ can be computed by an algebraic formula, which is essentially a relative version of the Gonz\'{a}lez-Sprinberg formula for local Euler obstructions. We give the relative Nash and relative conormal version of this formula in Theorem \ref{nash-integral} and Theorem \ref{mu-chi-conor}. 

The reader will see that the theory does not depend on the smoothness of $M$, but only depends on the generic smoothness of $f$ and the assumption ``sans \'{e}clatement en codimension 0". So the theory will work as fine if we replace $M$ by any reduced analytic space. We state our theorems for a complex manifold $M$ because it is most relevant to our major application: to study the characteristic cycle associated with the deformation of an ICIS.

Because the existence of the Milnor fibration is guaranteed by the condition no blowup in codimension 0, we may consider a Milnor number constructible function $\mu$. Recall that given $z \in M$, the Milnor fibre at $z$ is given by
\begin{equation*}
F_z = B_{\epsilon}(z)\cap f^{-1}(c)
\end{equation*}
where $B_{\epsilon}(z)$ is a sufficiently small real ball centred at $z$ and $c \in N$ is sufficiently close to but not equal to $f(z)$. The Milnor number is defined to be
\begin{equation*}
\mu(z) = (-1)^{m-n}\Big(\chi\big(F_z\big)-1\Big)
\end{equation*}
where $m$ ($n$) is the dimension of $M$ ($N$). Equivalently we can write 
\begin{equation}\label{mudef}
1 = \chi(z) + (-1)^{m-n+1}\mu(z).
\end{equation}

The function $\mu$ assigns each $z \in M$ the integer $\mu(z)$. It is a constructible function with support inside the critical space of $f$:
\begin{equation*}
C(f)= \{ z \in M \mid \rank df_z < n \}.
\end{equation*} 
 
It follows from the theory of Chern class transformation that $\mu$ has a characteristic cycle $\Ch(\mu)$. It contains very rich information about the degeneraction of $f$. To find $\Ch(\mu)$, we consider the vector bundle morphism $df: f^{*}T^{*}N \to T^{*}M$, and perform MacPherson's graph construction (\cite{MR732620} \S 18.1) for this morphism. We will show that the cycle in the limit of the graph construction can be grouped into two parts. One part is the Nash modification relative to $f$, hence can be used to compute $\chi(F_z)$ for any $z \in M$ by the Gonz\'{a}lez-Sprinberg type formula mentioned earlier; the other part can be used to compute $\mu(z)$ for any $z \in M$, as the consequence of a homotopy invariance argument.
The one which computes $\mu(z)$ is the projectivisation of a conic Lagrangian cycle, therefore it is the projectivised characteristic cycle of $\mu$. These results are also recorded in Theorem \ref{nash-integral} and Theorem \ref{mu-chi-conor}.

Since one can always perform MacPherson's graph construction as long as $f$ is holomorphic, there is an interesting byproduct of our theory. We can always define two constructible functions $\chi$ and $\mu$ using the algebraic formula given by Theorem \ref{nash-integral} (or Theorem \ref{mu-chi-conor}). They always satisfy equation \eqref{mudef}. When the additional flatness and no blowup in codimension 0 conditions are met, we can say $\chi(z) = \chi(F_z)$ and $\mu(z)$ is the usual Milnor number.

When $N = \mathbb{C}^n$, we may consider a family of embeddings $i_t: M \to M \times \mathbb{C}^n$ given by 
\begin{equation*}
z \in M \mapsto (z,-tf(z)).
\end{equation*}
Let $M_t$ be the image of $i_t$. One can see that the family of graphs appearing in the graph construction is equivalent to the family $T^{*}_{M_t}(M \times N)$ when $t \neq \infty$. This is explained in \S \ref{deformation} for $n=1$, but there is no essential difference when $n > 1$. We also compute the limit of the conormal family explicitly in \S \ref{computelag} in order to clarify the relation between these two deformations. The conclusion we wish to convey is that, MacPherson's graph construction for $f^{*}T^{*}N \to T^{*}M$ can be viewed as the globalisation of the locally defined Lagrangian specialisations.

As one main application of our theory, we will prove a formula of Toru Ohmoto concerning the Chern class transformation of $\mu$. Since our research is strongly motived by Ohmoto's result, which we received from him by grace in the form of personal manuscript, we feel obligated (but willingly) to say a few words about his remarkable formula.

A holomorphic map $f: M \to N$ is said to have finite contact type if any singularity of $f$ is an ICIS. Such map is flat and has no blowup in codimension 0 (Proposition \ref{nbl}). Therefore it makes sense to consider the constructible function $\mu$ under this setup.

Over $\mathbf{P}(f^*T^*N)$, we have an exact sequence
\begin{equation*}
0 \to \xi_N \to \pi_N^{*}f^{*}T^*N \to \zeta_N \to 0
\end{equation*}
where $\pi_N: \Pbf(f^*T^*N) \to M$ is the canonical projection and $\xi_N$ is the tautological line bundle.

Composing $\xi_N \to \pi_N^{*}f^{*}T^*N$ and $\pi_N^{*}f^{*}T^*N \to \pi_N^{*}T^*M$, we get a morphism $\xi_N \to \pi_N^{*}T^*M$. Twisting by $\xi_N^{\vee}$, we obtain a section $\Osr_{\Pbf(f^*T^*N)} \to \pi_N^{*}T^*M \otimes \xi_N^{\vee}$ of $\pi_N^{*}T^*M \otimes \xi_N^{\vee}$. Let $Z \subset \Pbf(f^*T^*N)$ be the zero locus of this section. When $f$ is Thom-Boardman-generic (TB-generic in the rest of the paper), $Z$ is a complex manifold representing $c_m(\pi_N^{*}T^*M \otimes \xi_N^{\vee})$, and $\pi_N\vert_Z: Z \to C(f)$ is a desingularisation of $C(f)$ \cite{TO}. 

Ohmoto proved the following theorem by Thom polynomial theory:

\begin{theorem}\label{chern-mu}\cite{TO}
For $f: M \to N$ which is TB-generic and of finite contact type, it holds that 
\begin{equation}
c_*(\mu) = \pi_{N*}(c(\zeta_N^{\vee}) \cap [Z])) \in H_*(M).
\end{equation}
\end{theorem}

The left side of the equation is the Chern class transformation of $\mu$. 

We will give another proof in \S\ref{application}, based on our understanding of $\mathbf{\Ch}(\mu)$. Interestingly, a version of our $Z$ is also used in a recent paper of Aluffi, as a middle step for studying the Chern-Schwartz-MacPherson class of any embeddable scheme and the Milnor class of that scheme in case it is a complete intersection (See the description of $\mathscr{Y}$ in \cite{Aluffi} \S 2.2). We believe there must be certain common truth behind the coincidence of the constructions, but is not yet able to bring that common truth into light.

To talk about $\mathbf{\Ch}(\mu)$, we first need a manifold containing $\supp{\mu} = C(f)$, and the natural candidate is $M$. Then no matter what $\Ch (\mu)$ is, it must be a conic Lagrangian cycle living in $T^*M$, consequently $\Pbf(\Ch (\mu))$ must live in $\Pbf(T^*M )$. However, the cycle $[Z]$ appearing in Theorem \ref{chern-mu} lives in $\Pbf(f^{*}T^*N)$. To relate $\Pbf(\Ch (\mu))$ with $Z$, one considers the graph of $df: f^*T^*N \to T^*M$, and MacPherson's graph construction allows one to deform the graph either to $f^*T^*N$ or $T^*M$. It turns out that the use of $M$ as the smooth ambient space to contain $C(f)$ is not what we end up with eventually, but the deformation of graph is the key ingredient in obtaining an explicit expression for $\Ch (\mu)$.

The case $n =1$ was very well studied in the past two decades. We will first examine our strategy on this well understood ground. In this process, we will gather useful experiences which will carry us further in the general case. We will also obtain new proofs for known formulas about the characteristic cycles of hypersurfaces. This project is done in \S \ref{cchs}.
 
\vspace{0.5cm} 
\noindent
{\bf Acknowledgements}: First of all, I want to express a deep gratitude to Toru Ohmoto for his altruistic sharing of his manuscript \cite{TO} with me and inviting me twice to Hokkaido university to discuss topics around Theorem \ref{chern-mu}. Without his help, the research shown in the present paper cannot even be started. I also want to give sincere thanks to Professor Aluffi for his constant encouragement throughout years, and to James Fullwood for helpful discussions in multiple occasions. Finally, the key idea of this paper was conceived during a long and psychologically confusing period due to the uncertainty of employment situation. I want to thank God for the emotional support and divine friendship I received from my church in Korea.

\section{characteristic cycles of hypersurfaces}\label{cchs}

\subsection{}\label{known}
Hypersurfaces often arise in geometry in the following ways,
\begin{enumerate}
\item the hypersurface is defined by a holomorphic function $f$ on an open subset $M$ of $\Cbb^m$ (local case);
\item the hypersurface is defined by a section $s$ of a holomorphic line bundle $L$ on $M$ (global case);
\item the hypersurface is defined by $f^{-1}(y)$ where $f: M \to N$ is a holomorphic map, $y\in N$ and $\dim N =1$ (deformation of a hypersuface singularity). 
\end{enumerate} 

The characteristic cycles of hypersurfaces were well studied in case (1) and (2). I don't know any reference where (3) is explicitly studied. The reason for such vacancy is that, for most applications considered in the past, the interest is local on $N$ or $M$, i.e. about local deformation of hypersurface singularities or local Milnor fibrations, so that (3) and (1) make no difference in this local setup. Let me recall the well known results in the first two cases.

\begin{proposition}\label{local}\cite{MR1795550}
Let $(z_1,\ldots,z_m)$ be holomorphic coordinates on $M \subset \Cbb^m$. Let $X$ be the hypersurface defined by $f=0$, let $Y$ be defined by $(\frac{\p f}{\p z_1}, \ldots, \frac{\p f}{\p z_m})$ and let $\pi: \textup{Bl}_{Y}M \to M$ be the canonical projection. Note that there is a closed embedding of $\textup{Bl}_{Y}M$ into $\Pbf(T^*M)$. Denote the total transformation of $X$ under this blow-up by $\mathscr{X}$, and denote the exceptional divisor of $\pi$ by $\mathscr{Y}$. We have
\begin{enumerate}[(i)]
\item $[\Pbf \Ch (1_X)] = (-1)^{m-1}([\mathscr{X}]-[\mathscr{Y}])$;
\item $[\Pbf \Ch(\chi')] = (-1)^{m-1}[\mathscr{X}]$;
\item $[\Pbf \Ch(\mu)] = [\mathscr{Y}]$.
\end{enumerate}
\end{proposition}

\begin{remark}
In general, $\mathscr{Y}$ is not a analytic subspace of $\mathscr{X}$. However, $[\mathscr{X}] - [\mathscr{Y}]$ always have positive coefficients.
\end{remark}

\begin{remark}\label{mu-def}
In (ii), $\chi'$ is the constructible function defined by the Euler characteristic of the Milnor fibre. Here we think $\chi'$ as a constructible function on $X$. In other words, we only concern the Milnor fibres for points lying on $X$. However, later in this section, we will shift our perspective, and will study primarily another constructible function $\chi$ whose domain is $M$!

\begin{equation*}
\chi(p) = \chi\Big(B_{\epsilon}(p) \cap f^{-1}(q)\Big) 
\end{equation*}
where $p \in M$ and $q \in \Cbb$ is sufficiently close to $f(p)$. By definition $\chi(p) = \chi'(p)$ if $p \in Y$, and $\chi (p) = 1$ if $p \notin Y$. Note that $\chi - 1_M = (-1)^{m-1}\mu$ as constructible functions on $M$.
\end{remark}

\begin{proposition}\label{global}\cite{MR1795550}
Let $X$ be the hypersurface defined by a section $s$ of $L$. Let $Y'$ be the analytic subspace locally defined by $(f,\frac{\p f}{\p z_1}, \ldots, \frac{\p f}{\p z_m})$, where $f$ is an expression of $s$ in a local trivialisation of $L$. Denote the total transformation of $X$ under the blow-up $\pi: \textup{Bl}_{Y'}M \to M$ by $\mathscr{X'}$, and the exceptional divisor by $\mathscr{Y'}$. Then the projectivised characteristic cycles of the constructible functions $1,\chi',\mu$ take the same shape as those appearing in the proposition \ref{local}, provided we replace $\mathscr{X}$ and $\mathscr{Y}$ by $\mathscr{X'}$ and $\mathscr{Y'}$.
\end{proposition}

\begin{remark}
In the case of Proposition \ref{local} we have

\begin{enumerate}[(i)]

\item $\displaystyle c_*(1_X) = c(TM\vert_{X}) \cap \pi_*(\frac{[\mathscr{X}] - [\mathscr{Y}]}{1+\mathscr{X}-\mathscr{Y}})$

\item $\displaystyle c_*(\chi') = c(TM\vert_{X}) \cap \pi_*(\frac{[\mathscr{X}]}{1+\mathscr{X}-\mathscr{Y}})$

\item $\displaystyle c_*(\mu) = (-1)^{m-1}c(TM\vert_{X}) \cap \pi_*(\frac{[\mathscr{Y}]}{1+\mathscr{X}-\mathscr{Y}})$

\end{enumerate}
while in the case of Proposition \ref{global} the same formulas hold provided we replace $\mathscr{X}$ and $\mathscr{Y}$ by $\mathscr{X'}$ and $\mathscr{Y'}$.

\end{remark}

\begin{remark}
We have a closed embedding $\textup{Bl}_{Y}M \to \Pbf(\mathcal{P}^{1}_M L)$, where $\mathcal{P}^{1}_M L$ is the bundle of principal parts of $L$ over $M$. There exists an exact sequence
\begin{equation*}
0 \to T^*M \otimes L \to \mathcal{P}^{1}_M L \to L \to 0,
\end{equation*}
which allows us to embed $\Pbf(T^*M)$ canonically in $\Pbf(\mathcal{P}^{1}_M L)$. One can see that $\mathscr{X}, \mathscr{Y} \subset \Pbf(T^*M)$ with regard to this embedding. Consequently we can legitimately regard $\mathscr{X}, \mathscr{Y}$ as projectivised conic Lagrangian cycles. The appearance of $\Pbf(\mathcal{P}^{1}_M L)$ is much more than an auxiliary construction. The natural relevance of $\Pbf(\mathcal{P}^{1}_M L)$ to the characteristic cycles in question will be explained in \S\ref{principal parts} and \S\ref{global deformation}.
\end{remark}

\subsection{}\label{deformation}
Let us focus on case $(1)$ of \S\ref{known} first. For a given $f: M \to \Cbb$, we consider a family of embeddings $i_t: M \to M \times \Cbb$ given by 
\begin{align*}
i_t: M& \to M \times \Cbb \\
z&\to(z,-tf(z)),  
\end{align*}
and the parameter $t$ of this family takes values in $\Cbb$. Equivalently, the isomorphic image $M_t := i_t(M)$ has the equation $y + tf(z) = 0$, where $y$ is the coordinate on the $\Cbb$ factor. The embedding $i_t$ gives us the conormal space $T^{*}_{M_t}(M\times \mathbb{C})$. The conormal direction at a given point $(z,y) \in M_t \subset M \times \Cbb$ is $dy + t(df)$.

On the other hand, we can consider the induced map on cotangent bundles $f^*T^*\Cbb \to T^*M$. The graph of this map is a rank one vector subbundle of  $f^*T^*\Cbb \oplus T^*M$. Each of its fibre is generated by $(f^*(dy),df) \in f^*T^*\Cbb \oplus T^*M \cong T^*(M \times \Cbb)\vert_{M_1}$. Moreover, when we deform the graph, following the procedure in MacPherson's graph construction (\cite{MR732620} chapter 18.1), the fibre of the deformed graph is generated by $(f^*(dy),tdf) \in f^*T^*\Cbb \oplus T^*M \cong T^*(M \times \Cbb)\vert_{M_t}$ where $t \in \Cbb$ is the deformation parameter. The deformed graph for the parameter $t$ will be denoted by $\Gamma_t$ in the rest of the paper.

The lesson we learn here is that, there is an isomorphism from the conormal space to the deformed graph
\begin{equation}\label{conorm-graph}
\begin{split}
T^{*}_{M_t}(M\times \mathbb{C}) &\to \Gamma_t \\
(z,t,dy + t df) &\mapsto (z,f^{*}dy + tdf),
\end{split}
\end{equation}
forgetting the $t$ coordinate. The family of conormal spaces contains slightly more information because it remembers how $M_t$ sits inside $M \times \Cbb$ as a family of subspaces.

For the sake of easy transition to case $(3)$ of \S\ref{known} later, we prefer not to trivialise $f^*T^*\Cbb$. We list some notations which we will use throughout this section.

\begin{enumerate}

\item For each $t \in \Pbf, t \neq \infty$, the manifold $M$ is embedded in $\Pbf(T^*M\oplus f^*T^*\Cbb)$ as $\Pbf(\Gamma_t)$. Denote by $\mathscr{M} \subset \Pbf(T^*M\oplus f^*T^*\Cbb) \times \Cbb$ the family of embeddings. We have $M \times \Cbb \cong \mathscr{M}$ and $\mathscr{M}_t = \Pbf(\Gamma_t) \subset \Pbf(T^*M\oplus f^*T^*\Cbb) \times \{t\}$ for $t \neq \infty$. Moreover, we have a map $\psi: \Pbf(T^*M\oplus f^*T^*\Cbb) \times \Cbb \to \Pbf(T^*(M \times \Cbb)) \times \Cbb$ defined by $(z,l,t) \mapsto (i_t(z),l,t)$ where $z \in M$ and $l$ is a line in the fibre of $T^*M\oplus f^*T^*\Cbb$ over $z$. Note that $l$ can also be regarded as a line in $T^{*}_{i_t(z)}(M \times \Cbb)$ so that this map is well defined. It is clear that $\psi(\mathscr{M}_t) = \mathbf{P}(T^{*}_{M_t}(M\times \mathbb{C}))$, giving the projectivised inverse of \eqref{conorm-graph}.

\item Let $\xi$ be the tautological (line) bundle on $\Pbf(T^*M\oplus f^*T^*\Cbb)$, and let $\pi$ be the canonical projection $\Pbf(T^*M\oplus f^*T^*\Cbb) \to M$. For $t \neq \infty$, the restriction of the left exact sequence $0 \to \xi \to \pi^*(T^*M\oplus f^*T^*\Cbb)$ on $\mathscr{M}_t \cong M$ is the left exact sequence $0 \to \Gamma_t \to T^*M\oplus f^*T^*\Cbb$. Let $\pr_1$ be the first projection $\Pbf(T^*M\oplus f^*T^*\Cbb) \times \Cbb \to \Pbf(T^*M\oplus f^*T^*\Cbb)$ and let $\Gamma =  \pr_1^*\xi \vert_{\mathscr{M}}$.

\end{enumerate}

The following diagram is a summary of the situation above.

\begin{displaymath}
\xymatrix{ \Gamma  \ar[r] \ar[d] & \pr_1^*\xi \ar[d]  \\
                 \mathscr{M} \ar[r] & \Pbf(T^*M\oplus f^*T^*\Cbb) \times \Cbb \ar[r]^{\psi} \ar[d]& \Pbf(T^*(M \times \Cbb)) \times \Cbb \ar[d] \\
                  & \Cbb \ar[r] & \Cbb
                  }
\end{displaymath}

The map $f^*T^*\Cbb \to T^*M$ fails to be left exact at the subspace $Y$ of $M$ defined by the ideal $(\frac{\p f}{\p z_1}, \ldots, \frac{\p f}{\p z_m})$. Let $\overline{\mathscr{M}}$ be the closure of $\mathscr{M}$ in $\Pbf(T^*M\oplus f^*T^*\Cbb) \times \Pbf^1$. The limit $\mathscr{M}_{\infty}$ is defined by the following Cartesian square 
\begin{displaymath}
\xymatrix{
                \mathscr{M}_{\infty} \ar[r] \ar[d] & \overline{\mathscr{M}} \ar[d] \\
                \{ \infty \} \ar[r]^j & \Pbf^1
}
\end{displaymath}
and the limiting cycle $[\mathscr{M}_{\infty}]$ is defined by $j^{!}[\overline{\mathscr{M}}]$. By \cite{MR732620} example 18.1.6 (d), we have 
\begin{equation}\label{limitgraph}
[\mathscr{M}_{\infty}] = [\textup{Bl}_Y M] +[\Pbf(C_YM\oplus f^*T^*\Cbb\vert_Y)],
\end{equation}
where $C_YM$ is the normal cone to $Y$ in $M$.

Similarly, we can consider $\overline{\psi(\mathscr{M})}$ in $\Pbf(T^*(M \times \Cbb)) \times \Pbf^1$. The Lagrangian specialisation at $\infty$ is denoted by $[\psi(\mathscr{M})_{\infty}]$. By definition, $[\psi(\mathscr{M})_{\infty}] = j^![\overline{\psi(\mathscr{M})}]$. We will prove later that if $f^{-1}(0)$ is the only singular fibre, then
\begin{equation}\label{limitlagrangian}
[\psi(\mathscr{M})_{\infty}] = [\mathbf{P}(C_YM \oplus f^*T^*\Cbb\vert_{Y})] + [\mathscr{X}' \times \Cbb],
\end{equation}
where $C_YM \oplus f^*T^*\Cbb\vert_{Y}$ is regarded as a cone over $Y \times \{0\}$. Recall Proposition \ref{global} that $\mathscr{X}'$ is the total transform of $X$ in $\textup{Bl}_{Y'}M$. We cannot expect $\mathscr{X} = \mathscr{X}'$ as analytic subspaces of $\mathbf{P}(T^{*}M)$ but we have $[\mathscr{X}] = [\mathscr{X}']$. This can be deduced from the fact that $\textup{Bl}_{Y'}M$ and $\textup{Bl}_{Y}M$ have the same normalisation. Therefore formula \eqref{limitlagrangian} can also be written as
\begin{equation*}
[\psi(\mathscr{M})_{\infty}] = [\mathbf{P}(C_YM \oplus f^*T^*\Cbb\vert_{Y})] + [\mathscr{X} \times \Cbb]
\end{equation*}

By proposition \ref{local} (applied to the hypersurface $X \times \Cbb$ in $M \times \Cbb$), $ [\mathscr{X} \times \Cbb]$ is the projectivised characteristic cycle of the function on $X \times \mathbb{C}$ given by$(z,t) \mapsto (-1)^{m}\chi(z)$. 

We believe the following figures will help the reader to visualise the deformation process and understand especially why there is a common piece appearing in both limits. It is also clear from these figures that the deformation to the normal cone (Remark 5.1.1 \cite{MR732620}) can be fused into our synthetic view for graph construction and Lagrangian specialisation.

\begin{figure}[!h]
\begin{tikzpicture}[allow upside down]
  \draw (-2,0) -- (2,0) node[right] {$M \times \{0\}$};
  \draw [->] (0,0) -- (0,3) node[left] {$\mathbb{C}$};
  \draw (0,0) .. controls (1,0) and (1.5,0.8) .. (2,2)
  \foreach \p in {30,45,65,80,90} {
   node[sloped,inner sep=0cm,above,pos=\p*0.01,
      anchor=south west,
      minimum height=(10)*0.04cm,minimum width=(10)*0.04cm]
      (N \p){}
     } node[right] {$M_1$}
     ;
      \foreach \p in {30,45,65,80,90} {
      \draw[-latex] (N \p.south west) -- (N \p.north west);
    }
    
  \draw (0,0) .. controls (-1,0) and (-1.5,0.8) .. (-2,2);
 
  \draw (0,0) .. controls (0.5,0) .. (1,3)
  \foreach \p in {30,60,70,80,90} {
   node[sloped,inner sep=0cm,above,pos=\p*0.01,
      anchor=south west,
      minimum height=(10)*0.04cm,minimum width=(10)*0.04cm]
      (N2 \p){}
     } node[right] {$M_2$}
     ;
      \foreach \p in {30,60,70,80,90} {
      \draw[-latex] (N2 \p.south west) -- (N2 \p.north west);
    }
  
   \draw (0,0) .. controls (-0.5,0) .. (-1,3);

\end{tikzpicture}
\caption{Lagrangian specialisation. Some conormal vectors are marked.}
\end{figure}
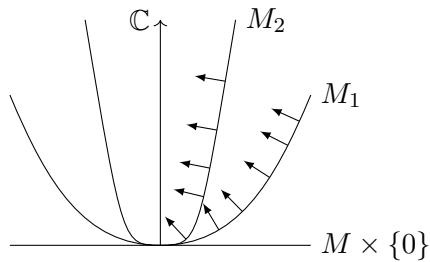

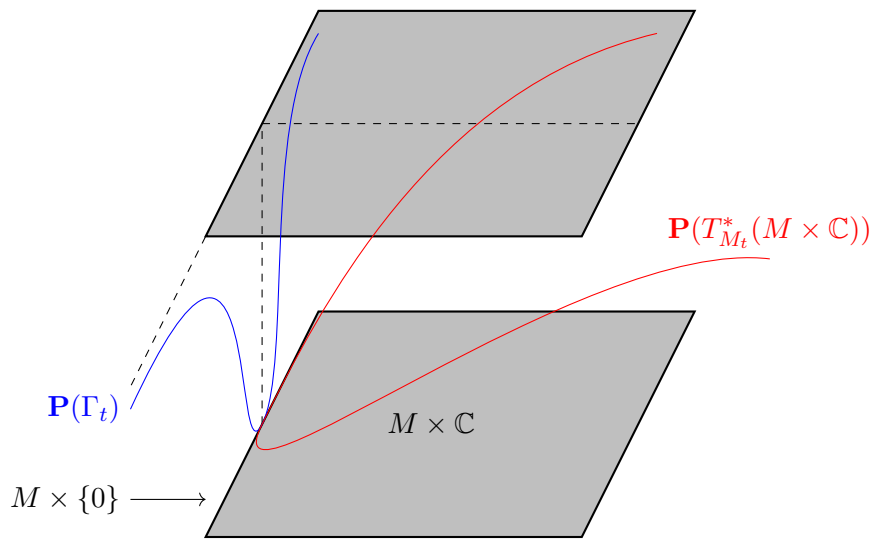
\begin{figure}[!h]
\begin{tikzpicture}
\coordinate (c1) at (0,0);
\coordinate (c2) at (5,0);
\coordinate (c3) at (6.5,3);
\coordinate (c4) at (1.5,3);
\coordinate (m1) at (0.75,1.5);

\coordinate (c5) at (0,4);
\coordinate (c6) at (5,4);
\coordinate (c7) at (6.5,7);
\coordinate (c8) at (1.5,7);
\coordinate (m2) at (0.75,5.5);
\coordinate (m3) at (5.75,5.5);

\coordinate (e2) at (1.5,6.7);
\coordinate (e1) at (-1,1.7);
\coordinate (f2) at (6,6.7);
\coordinate (f1) at (7.5,3.7);

\draw [thick,-,fill=lightgray] (c1) -- (c2) -- (c3) -- (c4) -- cycle;
\draw [thick,-,fill=lightgray] (c5) -- (c6) -- (c7) -- (c8) -- cycle;

\draw [dashed] (m1) -- (m2);
\draw [color=blue] (m1) .. controls +(0.4,0.8) and (m2) .. (e2);
\draw [dashed] (c5) -- (-1,2);
\draw [color=blue] (m1) .. controls +(-0.4,-0.8) and (m2) .. (e1) node(graph)[left] {$\mathbf{P}(\Gamma_t)$};

\draw [dashed] (m2) -- (m3);

\draw [color=red] (m1) .. controls (c4) and (3,6) .. (f2);
\draw [color=red] (m1) .. controls (c1) and (c6) .. (f1) node(conorm)[above] {$\mathbf{P}(T^{*}_{M_t}(M \times \mathbb{C}))$};

\node at (3,1.5) {$M \times \mathbb{C}$};

\draw [->] (-1,0.5) -- (0,0.5);
\node at (-1,0.5)[left] {$M \times \{0\}$};

\end{tikzpicture}
\caption{$\mathbf{P}(T^{*}_{M_t}(M \times \mathbb{C}))$ v.s. $\mathbf{P}(\Gamma_t)$. The vertical direction represents the cotangent directions of $M$.}
\label{f2}
\end{figure}

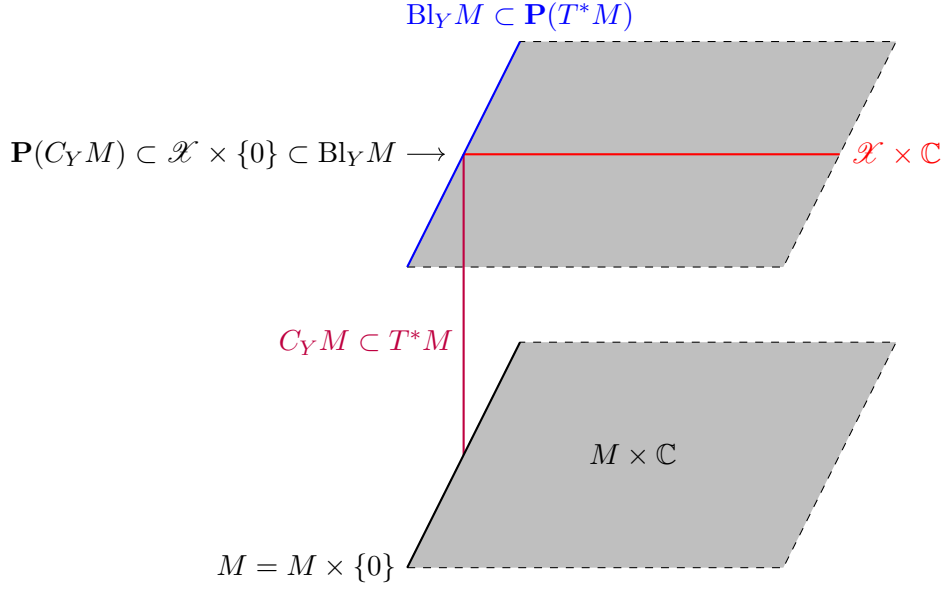
\begin{figure}[!h]
\begin{tikzpicture}
\coordinate (c1) at (0,0);
\coordinate (c2) at (5,0);
\coordinate (c3) at (6.5,3);
\coordinate (c4) at (1.5,3);
\coordinate (m1) at (0.75,1.5);

\coordinate (c5) at (0,4);
\coordinate (c6) at (5,4);
\coordinate (c7) at (6.5,7);
\coordinate (c8) at (1.5,7);
\coordinate (m2) at (0.75,5.5);
\coordinate (m3) at (5.75,5.5);

\draw [dashed,-,fill=lightgray] (c1) -- (c2) -- (c3) -- (c4) -- cycle;
\draw [dashed,-,fill=lightgray] (c5) -- (c6) -- (c7) -- (c8) -- cycle;

\draw [thick,color=red] (m2) -- (m3) node[right] {$\mathscr{X} \times \mathbb{C}$};
\draw [thick,color=purple] (m2) -- (m1);
\draw [thick,color=blue] (c5) -- (c8) node[above] {$\textup{Bl}_Y{M} \subset \mathbf{P}(T^*M)$};
\node at (0.75,3)[left,color=purple] {$C_YM \subset T^*M$};

\draw [thick] (c4) -- (c1) node[left] {$M = M \times \{0\}$};
\node at (3,1.5) {$M \times \mathbb{C}$};
\draw [->] (0,5.5) -- (0.5,5.5);
\node at (0,5.5)[left] {$\mathbf{P}(C_YM) \subset \mathscr{X} \times \{0\} \subset \textup{Bl}_Y{M}$};

\end{tikzpicture}
\caption{The deformation limit of $\mathbf{P}(T^{*}_{M_t}(M \times \mathbb{C}))$ v.s. the deformation limit of $\mathbf{P}(\Gamma_t)$. }
\label{f3}
\end{figure}

\begin{remark}
In general $|Y| \not \subset |X|$, i.e. there might be singular fibres other than $f^{-1}(0)$. Let $Y_0 = Y \setminus \{f \neq 0 \}$ be the singularities of $f$ belonging to $f^{-1}(0)$. The formula \eqref{limitlagrangian} for the limit of the family $\mathbf{P}(T^{*}_{M_t}(M \times \mathbb{C}))$ still holds, provided we replace $C_YM$ by $C_{Y_0}M$.
\end{remark}

\subsection{}\label{computelag}
Let us take a digression and finish the computation of the Lagrangian specialisation. We assume $f^{-1}(0)$ is the only singular fibre for simplicity. The figures in \S \ref{deformation} have already demonstrated vividly that the support of the limit projectivised Lagrangian cycle must be $|\mathscr{X} \times \mathbb{C}| \cup |\mathbf{P}(C_YM \oplus 1)|$. The question is only about the multiplicity of each irreducible component.

The homogeneous ring of $\psi(\mathscr{M})$ is 
\begin{equation*}
\mathscr{O}_{M \times \mathbb{C}}[\alpha,\beta_1,\ldots,\beta_m][t]/(y+tf(z),\ldots),
\end{equation*}
where $t$ is the deformation parameter; $y,z_1,\ldots, z_m$ are the coordinates on $M \times \mathbb{C}$ and $\alpha,\beta_1,\ldots,\beta_m$ are the coordinates for the corresponding cotangent directions. The ring is partially graded with $\deg(\alpha) = \deg(\beta_1) = \ldots, = \deg(\beta_m) = 1$. The `$\ldots$' consists of polynomials $h$ with coefficients in $\mathscr{O}_{M\times \mathbb{C}}$ and homogeneous in $\alpha, \beta_1,\ldots,\beta_m$, such that $h(1,t\frac{\p f}{\p z_1},\ldots,t\frac{\p f}{\p z_m}) = 0$.

To calculate the open subset of $\overline{\psi(\mathscr{M})}$ containing $\psi(\mathscr{M})_{\infty}$, we replace $t$ by $\frac{1}{s}$ in all relations defining $\psi(\mathscr{M})$, and then we multiply these new relations by powers of $s$ to clear denominators. Following this procedure, we see that the ideal of (the subset of) $\overline{\psi(\mathscr{M})}$ in $\mathscr{O}_{M \times \mathbb{C}}[\alpha,\beta_1,\ldots,\beta_m][s]$ has the following generators:
\begin{enumerate}[(a)]
\item $sy + f(z)$;
\item homogenous polynomials $h$ in $\mathscr{O}_{M \times \mathbb{C}}[\alpha,\beta_1,\ldots,\beta_m]$ such that $h(s,\frac{\p f}{\p z_1},\ldots,\frac{\p f}{\p z_m}) = 0$.
\item if a power of $s$ can be factored out from a linear combination of polynomials coming from (a) and (b), then dividing out this power of $s$, the remaining term is a generator of $\overline{\psi(\mathscr{M})}$.
\end{enumerate}

Let us introduce some notations first.

\begin{itemize}

\item Denote the complex line with coordinate function $y$ ($s$) by $\mathbb{C}_y$ ($\mathbb{C}_s$).
\item Let $\reallywidetilde{M \times \mathbb{C}_y \times \mathbb{C}_s}$ be the blow up of $M \times \mathbb{C}_y \times \mathbb{C}_s$ along $Y \times \mathbb{C}_y \times \{0\}$, and let $\phi:\reallywidetilde{M \times \mathbb{C}_y \times \mathbb{C}_s} \to M \times \mathbb{C}_y \times \mathbb{C}_s$ be the blowup map. Let $\reallywidetilde{M \times \mathbb{C}_s}$ be the blowup of  $M \times \mathbb{C}_s$ along $Y \times \{0\}$. Clearly $\reallywidetilde{M \times \mathbb{C}_y \times \mathbb{C}_s} = \reallywidetilde{M \times \mathbb{C}_s} \times \mathbb{C}_y$.
\item Let $A$, $B$, $C$, $D$ be the analytic subspaces of $M \times \mathbb{C}_y \times \mathbb{C}_s$ defined respectively by the ideals $\Big(sy + f(z)\Big)$, $\Big(sy + f(z),s\Big) = \Big(s,f(z)\Big)$, $\Big(s, f(z), \frac{\p f}{\p z_1},\ldots,\frac{\p f}{\p z_m}\Big)$ and $(s)$. Also let $Y'$ be the subspace of $M$ defined by $(f(z), \frac{\p f}{\p z_1},\ldots,\frac{\p f}{\p z_m})$. We have $X \times \mathbb{C}_y \times \{0\} = B$ and $Y' \times \mathbb{C}_y \times \{0\} = C$. 
\item Let $\tilde{A}$ ($\tilde{B}$) be the proper transform of $A$ ($B$) in $\reallywidetilde{M \times \mathbb{C}_y \times \mathbb{C}_s}$.

\end{itemize}

With these notations, It is clear that the generators in (b) defines $\reallywidetilde{M \times \mathbb{C}_y \times \mathbb{C}_s}$, and the open subset of $\overline{\psi(\mathscr{M})}$ containing $\psi(\mathscr{M})_{\infty}$ can be identified with $\tilde{A}$. Therefore $\psi(\mathscr{M})_{\infty} = \phi^{-1}(D) \cap \tilde{A} = \phi^{-1}(B) \cap \tilde{A}$. Moreover, if we identify $\tilde{A}$ with $\textup{Bl}_{C}A$, then $\phi^{-1}(B) \cap \tilde{A}$ can be identified with the total transform of $B$ in $\textup{Bl}_{C}A$. 

Our goal is to understand $[\phi^{-1}(D) \cap \tilde{A}]$. We claim that one part of $[\phi^{-1}(D) \cap \tilde{A}]$ is $[\mathscr{X}' \times \mathbb{C}_y]$. In fact, the open subset $U$ of $A$ defined by $y \neq 0$ is isomorphic to $M \times \mathbb{C}^{*}_y$ ($sy+f(z) = 0$ implies that $s = -\frac{1}{y}f(z)$), and $C \cap U$ is defined by $\Big(f(z), \frac{\p f}{\p z_1},\ldots,\frac{\p f}{\p z_m}\Big)$ so that $C \cap U$ is isomorphic to $Y' \times \mathbb{C}^{*}_y$. If we make the base change to $U$ for $\textup{Bl}_CA$, we then get an open subset of $\textup{Bl}_CA$ isomorphic to $\textup{Bl}_{Y'}M \times \mathbb{C}^{*}_y$. The ideal of $B$ in $U$ becomes $(f(z))$. So the total transform of $B$ is just $\mathscr{X'} \times \mathbb{C}^{*}_y$ in this open subset.

Next, let $\reallywidetilde{M \times \mathbb{C}_s} \times \mathbb{P}_y$ be the completion of $\reallywidetilde{M \times \mathbb{C}_s} \times \mathbb{C}_y$ along the $y$-axis, and let $\textup{pr}_1: \reallywidetilde{M \times \mathbb{C}_s} \times \mathbb{P}_y \to \reallywidetilde{M \times \mathbb{C}_s}$ be the first projection. Also let $\overline{\tilde{A}}$ be the closure of $\tilde{A}$ in $\reallywidetilde{M \times \mathbb{C}_s} \times \mathbb{P}_y$. The projection $\textup{pr}_1$ restricts to a proper modification $\overline{\tilde{A}} \to \reallywidetilde{M \times \mathbb{C}_s}$ because both spaces contain an open subset isomorphic to $M \times \mathbb{C}^{*}_s$. Finally, we let $D'$ be the principal divisor on $\reallywidetilde{M \times \mathbb{C}_s}$ defined by $(s)$. With the help of Figure \ref{f2} and \ref{f3}, it is not hard for one to conclude that 
\begin{equation*}
\textup{pr}^{*}_1(D') \cap [\overline{\tilde{A}}] = \gamma_1 + [\mathscr{X}' \times \mathbb{P}_y] + \gamma_2
\end{equation*}
where $\gamma_1$ is a cycle whose support is contained in $\mathbf{P}(C_YM \oplus 1)$ and $\gamma_2$ is a cycle whose support is contained in $\textup{Bl}_YM \times \{\infty_y\}$. In Figure \ref{f3}, the right edge of the parallelogram on the top can be thought of as $\textup{Bl}_YM \times \{\infty_y\}$.

It is clear that $\textup{pr}_{1*}$ restricts to the identity on $\gamma_1$ and $\gamma_2$, and it restricts to $0$ on $[\mathscr{X} \times \mathbb{P}_y]$. On the other hand, we have
\begin{equation*}
\textup{pr}_{1*}\Big(\textup{pr}^{*}_1(D') \cap [\overline{\tilde{A}}]\Big) = D' \cap [\reallywidetilde{M \times \mathbb{C}_s}] = [\mathbf{P}(C_YM \oplus 1)] + [\textup{Bl}_YM]
\end{equation*}
where the first equality uses the projection formula and the second equality follows from the deformation to the normal cone construction (\cite{MR732620} \S 5.1). Therefore we get $\gamma_1 = [\mathbf{P}(C_YM \oplus 1)]$ and $\gamma_2 = [\textup{Bl}_YM]$. Clearly $\gamma_2$ is the extra part coming from the completion, so we obtain
\begin{equation*}
[\phi^{-1}(D) \cap \tilde{A}] = \gamma_1 + [\mathscr{X}' \times \mathbb{C}_y] = [\mathbf{P}(C_YM \oplus 1)] + [\mathscr{X}' \times \mathbb{C}_y].
\end{equation*}

As what usually happens on the deformation limit, one should expect that there exists embedded components on $\psi(\mathscr{M})_{\infty}$. Formula \eqref{limitlagrangian} can only be true at the level of cycles. However, as explained in loc.cit., $\mathscr{M}_\infty$ can be regarded as the union of the subspaces $\textup{Bl}_Y M$ and $\Pbf(C_YM\oplus 1)$.

\subsection{}
The computation of the Lagrangian specialisation is not necessary for our discussion below. We have done the computation for the purpose of showing the relation and difference between the graph deformation and Lagrangian deformation. Because the 1-parameter deformation limit of s is still a , we can legitimately view $[C \oplus f^*T^*\Cbb\vert_{Y}]$ as a conic Lagrangian cycle  living in $(T^*M \oplus f^*T^*\Cbb)\vert_{Y} \cong T^*(M \times \Cbb)\vert_{Y\times \{0\}}$. We will show without using Proposition \ref{local} that it is the  of $\mu$ with regard to the embedding $M = M \times \{ 0 \} \to M \times \Cbb$. Here $\mu$ is treated as a constructible function on $M$ rather than $X$. Note that this perspective is natural for our final purpose, because when considering the deformation of a complete intersection singularity germ $f: (M,0) \to (N,0)$, the function $\mu$ can take non-zero values outside $f^{-1}(0)$.

To begin, we need to recall the process of associating a constructible function with a conic Lagrangian cycle . Given a complex manifold $M$ of dimension $m$, let $\pi_M: \Pbf(T^*M) \to M$ be the canonical projection and let
\begin{equation*}
0 \to \xi_M \to \pi_M^*T^*M \to \zeta_M \to 0
\end{equation*} 
be the standard sequence on $\Pbf(T^*M)$ defining the tautological subbundle $\xi_M$ (of rank $1$) and quotient bundle $\zeta_M$ (of rank $m-1$). From the isomorphism $\Pbf(T^*M) \cong \Pbb(T_M)$, we know that the dual sequence
\begin{equation*}
0 \to \zeta^{\vee}_M \to \pi_M^*(T_M) \to \xi^{\vee}_M \to 0
\end{equation*} 
defines the tautological subbundle $\zeta^{\vee}_M$ (of rank $m-1$) of hyperplanes for $\Pbb(T_M)$. If $\gamma$ is a conic Lagrangian cycle in $T^{*}M$, it is well-known that 
\begin{equation*}
\gamma = \sum_i k_i[T^{*}_{W_i}M]
\end{equation*}
for some analytic subvarieties $W_i \subset M$ and some integers $k_i$. The homology class associated to $\gamma$ is 
\begin{equation*} 
\sum_i (-1)^{m-1}k_i\pi_{M*}\Big(c(\zeta^{\vee}_M)\cap [\mathbf{P}(T^{*}_{W_i}M)]\Big),
\end{equation*} 
and the construction function associated with $\gamma$ is
\begin{equation*}
f_{\gamma}(z) = \sum_i \int (-1)^{m-1}k_ic(\zeta^{\vee}_M)\cap s\Big(\pi_M^{-1}(z) \cap \mathbf{P}(T^{*}_{W_i}M), \mathbf{P}(T^{*}_{W_i}M)\Big).
\end{equation*} 

Notably, when $W$ is a $d$-dimensional subvariety of $M$ and $\gamma = (-1)^d [T^{*}_WM]$, the homology class associated with $\gamma$ is the Chern-Mather class of $W$ and the constructible function associated with $\gamma$ is $\textup{Eu}_W$, the local Euler obstruction for $Z$.

In particular, if $V$ be a purely $m$-dimensional complex subspace of $T^*M$ whose associated cycle $[V]$ is Lagrangian, then homology class associated to $[V]$ is
\begin{equation*} 
(-1)^{m-1}\pi_{M*}\Big(c(\zeta^{\vee}_M)\cap[\Pbf(V)]\Big),
\end{equation*} 
which is also the dual class of $\pi_*(c(\zeta_M)\cap [\Pbf(V)])$. By a formula about the Segre class (\cite{MR732620} Lemma 4.2), the constructible function $f_V$ associated with $V$ can be written as
\begin{equation*}
f_V(z) = \int (-1)^{m-1}c(\zeta^{\vee}_M)\cap s\big(\pi_M^{-1}(z) \cap \Pbf(V), \Pbf(V)\big),
\end{equation*} 
where $z\in M$. However, we caution the readers that we {\bf{can't}} compute this value by
\begin{equation*}
f_V(z) = \int c(\zeta_M)\cap s\big(\pi_M^{-1}(z) \cap \Pbf(V), \Pbf(V)\big)
\end{equation*} 
because the Segre class $s\big(\pi_M^{-1}(z) \cap \Pbf(V), \Pbf(V)\big)$ is not a priori purely $(m-1)$-dimensional.

\subsection{}\label{microlocal}
Since the notations are getting complicated, let us recollect notations we have introduced earlier and define some new ones.
\begin{itemize}
\item $\xi, \tilde{\xi}, \xi_M$ are tautological subbundles of $\Pbf(T^*M \oplus f^*T^*\Cbb), \Pbf(T^*(M \times \Cbb)), \Pbf(T^*M)$ respectively. Similarly $\zeta, \tilde{\zeta}, \zeta_M$ are tautological quotient bundles of the corresponding spaces.
\item $\pi, \tilde{\pi}, \pi_M$ are canonical projections from $\Pbf(T^*M \oplus f^*T^*\Cbb), \Pbf(T^*(M \times \Cbb)), \Pbf(T^*M)$.
\item $p$ is the projection $\textup{Bl}_Y M \to M$.
\end{itemize}

According to our discussions by far, we wish to show
\begin{equation*}
\mu(z) = \int (-1)^{m}c(\tilde{\zeta}^{\vee})\cap s\Big(\tilde{\pi}^{-1}(z) \cap \Pbf(C_YM\oplus f^*T^*\Cbb\vert_{Y}),\Pbf(C_YM\oplus f^*T^*\Cbb\vert_{Y})\Big),
\end{equation*}
where $z \in M \times \{ 0 \} \subset M \times \Cbb$. The power of $-1$ is $m$ instead of $m-1$ because the nonsingular ambient space we choose for our embedding is $M \times \Cbb$, which has dimension $m+1$. We note that $T^*M \oplus f^*T^*\Cbb \cong \Omega^1_{M\times \Cbb}\vert_{M\times \{0\}}$. It follows that $\zeta \cong \tilde{\zeta}\vert_{\tilde{\pi}^{-1}(M \times \{0\})}$. Because $\Pbf(C_YM\oplus f^*T^*\Cbb\vert_{Y}) \subset \tilde{\pi}^{-1}(M \times \{0\})$, so equivalently we must show
\begin{equation*}
\mu(z) = \int (-1)^{m}c(\zeta^{\vee})\cap s\Big(\pi^{-1}(z) \cap \Pbf(C_YM\oplus f^*T^*\Cbb\vert_{Y}),\Pbf(C_YM\oplus f^*T^*\Cbb\vert_{Y})\Big).
\end{equation*}

By the definition of $\mu$ given in Remark \ref{mu-def} and equation \eqref{limitgraph}, it is enough to show
\begin{equation*}
1 = 1_M(z) = \int c(\zeta^{\vee})\cap s\big(\pi^{-1}(z) \cap \mathscr{M}_{\infty}, \mathscr{M}_{\infty}\big)
\end{equation*}
and 
\begin{equation}\label{chi}
\chi(z) = \int c(\zeta^{\vee})\cap s\big(\pi^{-1}(z) \cap \textup{Bl}_Y M, \textup{Bl}_Y M\big)
\end{equation}
for any $z \in M$.

Since we have a family $\overline{\mathscr{M}} \subset \Pbf(T^*M \oplus f^*T^*\Cbb) \times \Pbf^1$, the tautological bundle $\textup{pr}^{*}_1\zeta$ can be viewed as a family of bundles $\zeta_t = \zeta\vert_{\mathscr{M}_t}$. On each $\mathscr{M}_t \cong M$ ($t \neq \infty$), we have a short exact sequence
\begin{equation*}
0 \to \Gamma_t \to T^*M \oplus f^*T^*\Cbb \to \zeta_t \to 0.
\end{equation*}
In particular, we have $\zeta_0 \cong T^*M$, $\zeta\vert_{\Pbf(T^*M)} \cong \zeta_M \oplus \pi_M^*f^*T^*\Cbb$. Now, we have
\begin{align*}
 & \int c(\zeta_{\infty}^{\vee})\cap s\big({\pi}^{-1}(z) \cap \mathscr{M}_{\infty}, \mathscr{M}_{\infty}\big) \\
 = & \int c(\zeta_{0}^{\vee})\cap s\big({\pi}^{-1}(z) \cap \mathscr{M}_{0}, \mathscr{M}_{0}\big) \\
 = & \int c(TM) \cap s(z,M) \\
 = & 1
\end{align*}
by \cite{MR732620} Corollary 6.5 and example 4.1.6 (b). 

In \S \ref{complete intersections} we will establish a general formula for the Euler characteristic of the Milnor fibre for a map sans \'{e}clatment en codimension 0. Equation \eqref{chi} follows from Theorem \ref{nash-integral}. Hence we see from this easy deformation argument that $\mathbf{P}(C_YM\oplus f^*T^*\Cbb\vert_{Y})$ is the projectivised characteristic cycle of $\mu$ with respect to the embedding $M \times\{0\} \to M \times \mathbb{C}$.

\begin{remark}
The following is one attempt to show equation \eqref{chi}.
\begin{align*}
& \int c(\tilde{\zeta}^{\vee})\cap s\big(\pi^{-1}(z) \cap  \textup{Bl}_Y M, \textup{Bl}_Y M\big) \\
=& \int c(\zeta_M^{\vee}) c(p^*f^*T\Cbb) \cap s\big(p^{-1}(z), \textup{Bl}_Y M\big) \\
=& \int c(\zeta_M^{\vee}) \cap s\big(p^{-1}(z), \textup{Bl}_Y M\big),
\end{align*}
and if we had 
\begin{equation}\label{q1}
s\big(p^{-1}(z), \textup{Bl}_Y M\big) = s\big(p^{-1}(z), \mathscr{X}\big),
\end{equation}
we would obtain the correct result according to proposition \ref{local} (ii). Note that equation \eqref{q1} indeed has a plausible shape, because the normal bundle to $\mathscr{X}$ in $\textup{Bl}_Y M$ is trivial. Unfortunately, the example $M = \Cbb^2$ and $f=xy$ shows that equation \eqref{q1} is wrong. 

Indeed, in this example, $p^{-1}(0)$ is the exceptional divisor $E$ of the blowup.  We have $s\big(E, \textup{Bl}_Y M\big) = c(\mathscr{O}(E))^{-1} \cap [E] = [E] +[pt]$ and $s\big(E, \mathscr{X}\big) = 2[E] + 2[pt]$. 

On the other hand, we have $\mathbf{P}(T^{*}M) \cong \mathbb{C}^2 \times \mathbb{P}^1$, $\zeta_M \cong \mathscr{O}(1)$, so $\zeta_M \vert_E \cong \mathscr{O}_E(-E)$. And 
\begin{equation*}
\int c(\zeta_M^{\vee}) \cap s\big(p^{-1}(z), \textup{Bl}_Y M\big) = \int c(\mathscr{O}_E(E)) \cap ([E] + [pt]) = 0.
\end{equation*}
Similarly,
\begin{equation*}
\int c(\zeta_M^{\vee}) \cap s\big(p^{-1}(z), \mathscr{X}\big) = 0.
\end{equation*}
Therefore we have
\begin{equation*}
\int c(\zeta_M^{\vee}) \cap s\big(p^{-1}(z), \textup{Bl}_Y M\big) = \int c(\zeta_M^{\vee}) \cap s\big(p^{-1}(z), \mathscr{X}\big)
\end{equation*}
in this example. In Theorem \ref{nash-integral}, we will show that such equality always holds when $\dim N = 1$. So our method will provide an alternative proof for the fact that $(-1)^{m-1}[\mathscr{X}]$ is the projectivised characteristic cycle for $\chi'$. However, this seems to be a special phenomenon for hypersurfaces. When $\dim N > 1$, we cannot deduce from our method that $p^{-1}_{\tilde{M}}(M_{f(z)})$ computes $\chi(z)$ (in the notation of Theorem \ref{nash-integral}). See also \cite{MR732620} Example 4.2.7 and 4.2.8 for a discussion on when an equality of type \eqref{q1} can be correct. 

\end{remark}

\subsection{}\label{misc}

We can apply the same technic to the equation \eqref{limitlagrangian}. Having seen that 
\begin{itemize}
\item $[\mathscr{X} \times \Cbb]$ is the characteristic cycle of the function $(z,t) \mapsto (-1)^m\chi(z)$ for $z\in X$;
\item $[\mathbf{P}(C \oplus f^*T^*\Cbb\vert_{Y})]$ is the characteristic cycle of the function $(z,0) \mapsto \mu(z)$ for $z\in Y$;
\item $[\mathbf{P}(\Gamma_t)]$ is the characteristic cycle of the function $(-1)^m1_{M_t}$,
\end{itemize}
then equation \eqref{limitlagrangian} gives us the relation $(-1)^m\chi(z) + \mu(z) = (-1)^m$ where $z \in Y = Y \times \{0\}$. Again, the integration over $\Pbf(\Gamma_{\infty})$ can be turned into the integration over $\Pbf(\Gamma_{t})$ by deforming inside $\Pbf(T^*(M \times \Cbb))$.

We see that there are two deformation processes for $M$. The Lagrangian deformation is more or less what we ought to do, following the general theory of Lagrangian specialisation, and is closer to many former approaches to the characteristic cycles of hypersurfaces. Indeed, one may try to compute the ideal defining $\tilde{A}$ in \S \ref{computelag}. In doing so, one will quickly run into some arguments involving the integral dependence of $f$ on the ideal $(z_1\frac{\p f}{\p z_1},\ldots,z_m\frac{\p f}{\p z_m})$ (\cite{MR0568901} \S 2.7 exercise (3)).This integral dependence is at the root of the multiplicity calculation in former works such as \cite{MR1795550} and \cite{MR1819626}.

On the other hand, the graph deformation is much easier, and has certain advantages. For example, when we consider the map $f:M \to N$, where $N$ is not isomorphic to $\Cbb^n$, we can't form a family of embeddings $i_t: M \to M \times N$ as we did so far. Therefore the Lagrangian deformation is not immediately defined. However we can still deform the graph of  $f^*T^*N \to T^*M$. This is what we will do in \S\ref{complete intersections}.

Following the spirit of our achievement so far, we can conclude that in the case (3) of \S\ref{known}, the $\mu$ function can be computed by the following formula.
\begin{equation}\label{hypersurface-deform}
\mu(p) = (-1)^m\int c(\zeta^{\vee}) \cap s\Big(\pi^{-1}(p)\cap \Pbf(C \oplus f^*T^*N\vert_{Y}), \Pbf(C \oplus f^*T^*N\vert_{Y})\Big)
\end{equation}
where $\pi: \Pbf(T^*M \oplus f^*T^*N) \to M$ is the projection, $\zeta$ is the tautological quotient bundle of $\Pbf(T^*M \oplus f^*T^*N) \to M$, and $C, Y$ are explained below.

The image of the morphism $TM \otimes f^*T^*N \to \Osr_M$ is a coherent sheaf of ideal. We denote it by $\mathscr{J}$. Let $Y$ be the complex subspace determined by $\mathscr{J}$. Therefore we have a surjection of sheaves of algebras
\begin{equation*}
\textup{Sym}(TM) \to \textup{Rees}(\mathscr{J} \otimes f^*TN).
\end{equation*}
The ideal sheaf $\mathscr{J}$ in degree $0$ induces another surjection
\begin{equation*}
\textup{Sym}(TM)/\mathscr{J}\textup{Sym}(TM) \to \textup{Rees}(\mathscr{J} \otimes f^*TN)/\mathscr{J}\textup{Rees}(\mathscr{J} \otimes f^*TN)
\end{equation*}
which in turn defines the subcone $C$ of $\textup{Spec}(\textup{Sym}(TM\vert_Y))$. Note that $\textup{Spec}(\textup{Sym}(TM\vert_{Y}))$ is the total space of $T^*M\vert_{Y}$.

Note that formula \eqref{hypersurface-deform} does not provide us anything new, for after all, $Y$ is contained in several disconnected hypersurfaces and $f^*T^*N\vert_{Y}$ is trivial. However, the way we write down the formula will make it appear more consistent with further results in \S \ref{complete intersections}.

\subsection{}\label{principal parts}
Finally we deal with case (2) of \S\ref{known}. We wish to explain how to obtain the characteristic cycles in proposition \ref{global} from a global deformation. Our treatment of case (1) of \S\ref{known} attached two deformations to the section $f \in \Gamma(M,\mathscr{O}_M)$: the Lagrangian deformation and MacPherson's graph deformation. In the case (2) of \S\ref{known}, the Lagrangian deformation is quite obvious. Let $\mathscr{L}$ be the sheaf of sections of the line bundle $L$. The section $s \in \Gamma(M,\mathscr{L})$ allows us to view $M$ as a subspace of L, and we can deform $M$ inside $L$ by $ts$, where $t \in \Cbb$. The conormal spaces of this family of embeddings gives the family of conic Lagrangian cycles s in $T^*L$. However, it is not immediately clear what a substitute for the graph deformation should be. To find the correct geometric context for an analogue of the graph deformation, we first need to recollect some basic results for principal $G$-bundles when $G$ is a connected complex Lie group.

Let $P$ be a principal bundle over $M$ with group $G$. Then there exists a canonical exact sequence of vector bundles over $M$ (\cite{MR0086359} Theorem 1).
\begin{equation*}
0 \to L(P) \to Q \to TM \to 0,
\end{equation*}
where $L(P)$ is the bundle associated to $P$ by the adjoint representation of $G$, and $Q$ is the bundle of invariant vector fields on $P$. This short exact sequence determines an extension class $a(P) \in Ext^1(TM,L(P)) \cong H^1(M,T^*M \otimes L(P))$.

Let $E$ be a vector bundle over $M$. When $P$ is the frame bundle associated to $E$, it can be shown that $L(P) \cong \mathcal{End}(E)$ (\cite{MR0086359} proposition 9). In this case, the extension class $a(P)$ can be regarded as an element in $H^1(M,T^*M \otimes \mathcal{End}(E))$.

For a coherent sheaf $\mathscr{F}$ on $M$, there also exists an exact sequence (\cite{MR0086359} \S 4)
\begin{equation}
0 \to T^*M \otimes \mathscr{F} \to \mathcal{P}^1(\mathscr{F}) \to \mathscr{F} \to 0.
\end{equation}
Thus, given a vector bundle $E$, taking $\mathscr{F}$ to be the sheaf of sections of $E$ defines another extension class $b(E) \in Ext^1(E, T^*M \otimes E) \cong H^1(M,T^*M \otimes \mathcal{End}(E))$. 

These two extension classes for a vector bundle $E$ are related by $a(P) = -b(E)$ (\cite{MR0086359} Theorem 5).

If moreover $E = L$ is a line bundle, it is then clear that $\mathcal{End}(E) \cong \mathscr{O}_M$, and $P \cong L^{\times}$, the complement of the zero section in $L$. In this very special case, our first exact sequence takes the form
\begin{equation}\label{1stsq}
0 \to \mathscr{O}_M \to Q \to TM \to 0,
\end{equation}
and our second exact sequence, after taking the tensor product with $\mathscr{L}^{\vee}$ takes the form
\begin{equation}\label{2ndsq}
0 \to T^*M \to \mathcal{P}^1(\mathscr{L})\otimes \mathscr{L}^{\vee} \to \mathscr{O}_M \to 0.
\end{equation}
It is clear that the tensor product with a line bundle does not change the extension class of a short exact sequence. Therefore the extension class in $H^1(M, T^*M)$ determined by \eqref{2ndsq} is again $b(E)$. Using \cite{MR0086359} proposition 3, we see that

\begin{proposition}[implicitly stated in \cite{MR0086359}]
The short exact sequences \eqref{1stsq} and \eqref{2ndsq} are dual to each other. In particular, $\mathcal{P}^1(\mathscr{L})\otimes \mathscr{L}^{\vee}$ can be regarded as the bundle (over $M$) of invariant forms on the principal bundle $L^{\times}$.
\end{proposition}

\begin{remark}
The definition of $\mathcal{P}^1(\mathscr{F})$ given in \cite{MR0086359} \S 4 is different from the commonly accepted one as in \cite{MR1308020}. Here we show their equivalence. 

When $X$ is a separated scheme over the ground field $k$, we let $\Delta_{(1)}$ be the first infinitesimal neighbourhood of the diagonal in $X \times X$ and $\pi_1,\pi_2$ be the restrictions of the two projections to $\Delta_{(1)}$. Let $\mathscr{F}$ be a coherent sheaf on $X$. The sheaf of the 1st order principal parts $\mathcal{P}^1(\mathscr{F})$ is defined by $\pi_{1*}\pi^*_2\mathscr{F}$ in \cite{MR1308020}. Very concretely, if we denote the ideal sheaf of the diagonal by $\mathscr{I}$, then 
\begin{equation*}
\mathcal{P}^1(\mathscr{F}) = \Big(\mathscr{O}_{X \times X} \slash \mathscr{I}^2\Big) \otimes \mathscr{F},
\end{equation*}
where the tensor product uses the right $\mathscr{O}_X$-module structure of $\mathscr{O}_{X \times X}$, and the $\mathscr{O}_X$-module structure of $\mathcal{P}^1(\mathscr{F})$ is inherited from the left $\mathscr{O}_X$-module structure of $\mathscr{O}_{X \times X}$.

We also have the right (left) $\mathscr{O}_X$-module inclusion $i_r$ ($i_l$): $\mathscr{O}_X \to \mathscr{O}_{X \times X}$, locally given by $i_r(f) = 1 \otimes f$ ($i_l(f) = f \otimes 1$), where $f$ is a local section of $\mathscr{O}_X$. The inclusion $i_r$ induces $i_r \otimes \mathscr{F}$: $\mathscr{F} \to \mathcal{P}^1(\mathscr{F})$, which we still denote by $i_r$ for simplicity. In \cite{MR1308020}, this morphism is denoted by $d^1$, and an interpretation of this morphism in terms of taking the truncated Taylor expansion is given there. One can see that $i_r$ is $k$-linear, but not $\mathscr{O}_X$-linear. It is this morphism that gives the $\mathbb{C}$-splitting of $\mathcal{P}^1(\mathscr{F})$ in the definition of $\mathcal{P}^1(\mathscr{F})$ given in \cite{MR0086359} \S 4. Moreover, let $s$ be a local section of $\mathscr{F}$, then we have 
\begin{align*}
fi_r(s) &= f \otimes 1 \otimes s \\
            & = (f \otimes 1 - 1 \otimes f) \otimes s + 1 \otimes 1 \otimes fs \\
            & = -df \otimes s + i_r(fs), 
\end{align*}
agreeing with the description of the $\mathscr{O}_X$-structure of $D(S)$ given in \cite{MR0086359} \S 4. (Setting $\beta = 0$ in the formula (ii) there. The negative sign in front of $df \otimes m$ is insignificant, and is due to the convention we use for defining $df$.)
\end{remark}

We will only consider $\mathscr{F} = \mathscr{L}$ the sheaf of sections of a line bundle $L$ until the end of this section. Let $\{ U_i\}$ be a cover of $M$ such that there exists local trivialisations $u_i: \mathscr{L} \vert_{U_i} \to \mathscr{O}_{U_i}$. Let $e_i = u^{-1}_i(1)$ be the local frames, and let $\tau_{ji} = u^{-1}_j u_i : \mathscr{L} \vert_{U_i \cap U_j} \to \mathscr{L} \vert_{U_i \cap U_j}$ be the chart transitions. Let $g_{ji} \in \Gamma (U_i \cap U_j, \mathscr{O}^{*}_M)$ such that $\tau_{ji} (e_i) =  e_j = g_{ji} e_i$.

We wish to work out explicitly the chart transition law for $\mathcal{P}^1(\mathscr{L})$. For this, we first notice that locally we can define $\mathscr{O}_{U_i}$-splitting morphism with the help of $i_l$. Namely, we define $i_{l,i}: \mathscr{L} \vert_{U_i} \to \mathcal{P}^1(\mathscr{L}) \vert_{U_i}$ by the composition
\begin{equation*}
\mathscr{L} \vert_{U_i} \xrightarrow{u_i} \mathscr{O}_{U_i} \xrightarrow{i_l} \mathscr{O}_{U_i \times U_i} \xrightarrow{1 \otimes u^{-1}_i} \mathscr{O}_{U_i \times U_i} \otimes \mathscr{L} \vert_{U_i}.
\end{equation*}
In other words, it is $fe_i \mapsto f\otimes 1 \otimes e_i$ for a local section $f \in \Gamma (U_i, \mathscr{O}_M)$. To summarise, we have a globally defined $\mathscr{O}_M$-linear inclusion $j: T^*M \otimes \mathscr{L} \to \mathcal{P}^1(\mathscr{L})$, a globally defined $\mathbb{C}$-linear inclusion $i_r: \mathscr{L} \to \mathcal{P}^1(\mathscr{L})$, and a locally defined $\mathscr{O}_M$-linear inclusion $i_{l,i}: \mathscr{L} \vert_{U_i} \to \mathcal{P}^1(\mathscr{L}) \vert_{U_i}$. The morphisms $j$ and $i_{l,i}$ give the local $\mathscr{O}_{U_i}$-splitting of $\mathcal{P}^1(\mathscr{L}) \vert_{U_i}$. The three morphisms are related by $i_{r} - i_{l,i} = ju^{-1}_idu_i$. One should also note that $\nabla_i = u^{-1}_idu_i: \mathscr{L} \vert_{U_i} \to \Omega^1_{U_i} \otimes \mathscr{L}$ are locally defined Koszul connections.

Let $\alpha \in \Gamma (U_i \cap U_j, \mathcal{P}^1(\mathscr{L}))$, and $\alpha = j(w) + i_{l,i}(s)$ be a $\mathscr{O}_{U_i}$-splitting of $\alpha$ in the chart $U_i$, where $w \in \Gamma (U_i \cap U_j,T^*M\otimes \mathscr{L})$ and $s \in \Gamma (U_i \cap U_j,\mathscr{L}$). Now
\begin{align*}
\nabla_i(s) - \nabla_j(s) &= u^{-1}_idu_i(s) - u^{-1}_{j}du_j(s) \\
                                      &= \tau^{-1}_{ji}u^{-1}_jdu_j\tau_{ji}(s) - u^{-1}_{j}du_j(s) \\
                                      &= g^{-1}_{ji}\nabla_j(g_{ji}(s)) - \nabla_j(s) \\
                                      & = \frac{dg_{ji}}{g_{ji}} \otimes s.
\end{align*}
Therefore
\begin{align*}
\alpha & = j(w) + i_{l,i}(s) \\
           & = j(w) + i_{r}(s) - j(\nabla_i(s)) \\
           & = j(w) + i_r(s) - j(\nabla_j(s)) - j(\frac{dg_{ji}}{g_{ji}} \otimes s) \\
           & = j(w - \frac{dg_{ji}}{g_{ji}}\otimes s) + i_{l,j}(s).
\end{align*}

In other words, the transition from chart $i$ to chart $j$: $\Big((T^*M\otimes \mathscr{L}) \oplus \mathscr{L} \Big)\vert_{U_i \cap U_j} \to \Big((T^*M\otimes \mathscr{L}) \oplus \mathscr{L} \Big)\vert_{U_i \cap U_j}$ is given by $(w,s) \mapsto (w- \frac{dg_{ji}}{g_{ji}}\otimes s, s)$.

\begin{remark}
If we interpret $\alpha, w, s$ as sections over $U_i \cap U_j$ of the bundles $\mathcal{P}^1(\mathscr{L}) \otimes \mathscr{L}^{\vee}, T^*M, \mathscr{O}_M$ respectively, then the formula above also gives the chart transition law for $\mathcal{P}^1(\mathscr{L}) \otimes \mathscr{L}^{\vee}$ in terms of the local splittings $\mathcal{P}^1(\mathscr{L}) \otimes \mathscr{L}^{\vee} \vert_{U_i} \cong \Omega^1_{U_i} \oplus \mathscr{O}_{U_i}$.
\end{remark}

Let us compute the transformation law for the invariant forms on $L^{\times}$. Clearly $L^{\times}$ also trivialises over $U_i$, equivalently we have $u_i : L^{\times} \vert_{U_i} \xrightarrow{\sim} U_i \times \Cbb^{*}$. Let $t_i$ be the coordinate function along the $\Cbb$ factor in this trivialisation. Let $\alpha_{ji}$ be the composition $u_ju^{-1}_i: (U_i \cap U_j) \times \Cbb^* \to (U_i \cap U_j) \times \Cbb^*$. One can see that $\alpha_{ji}^{*}t_j = g_{ij}t_i$. Therefore $\alpha_{ji}^*(dt_j) = d(g_{ij}t_i) = t_i(dg_{ij}) + g_{ij}(dt_i)$, or equivalently
\begin{equation*}
\frac{dt_i}{t_i} = -\frac{dg_{ij}}{g_{ij}} + \alpha_{ji}^{*}\frac{dt_j}{t_j} = \frac{dg_{ji}}{g_{ji}} + \alpha_{ji}^{*}\frac{dt_j}{t_j}.
\end{equation*}

Since the form $\frac{dt_i}{t_i}$ and $\Omega^1_{U_i}$ clearly generate the $\mathbb{C}^*$ invariant forms on $L^{\times} \vert_{U_i}$, our local computations show that the assignment $(w, \frac{dt_i}{t_i}) \mapsto (-w,1)$ where $w$ is a section of $\Omega^1_{U_i}$ glues together to give an explicit isomorphism $Q^{\vee} \xrightarrow{\sim} \mathcal{P}^1(\mathscr{L}) \otimes \mathscr{L}^{\vee}$. It is also clear from the local description that $Q^{\vee} \cong \Omega^1_L(\log M) \vert_M$, where $M$ is embedded as a smooth divisor in $L$ by the zero section of $L$.

\subsection{}\label{global deformation}
Given a section $s \in \Gamma(M,\mathscr{L})$, we need to consider $T^*L \vert_{s(M)}$ in order to speak about the conormal space of $s(M)$ in $L$. Let $\pi: L \to M$ be the canonical projection. There exists an exact sequence
\begin{equation*}
0 \to \pi^*\mathscr{L} \to TL \to \pi^*TM \to 0.
\end{equation*}
The inclusion $s(M) \subset L$ induces $0 \to Ts(M) \to TL \vert_{s(M)}$. The subbundles $\pi^*\mathscr{L} \vert_{s(M)}$ and $Ts(M)$ of $TL \vert_{s(M)}$ split $TL \vert_{s(M)}$. Dually, we have
\begin{align*}
T^*L \vert_{s(M)} &= T^*s(M) \oplus \pi^*\mathscr{L}^{\vee} \vert_{s(M)} \\
s^*T^*L &\cong T^*M \oplus \mathscr{L}^{\vee}.
\end{align*}

In contrast to the case (1) of \S \ref{known}, where the function $f$ defines a conormal vector $dt + df$ at each point of the graph $M_1$, here we can't get a conormal vector at each point of $s(M)$. In fact, letting the local equation of $s$ over $U_i$ be $f_i$ through the trivialisation $u_i: \Gamma(U_i,\mathscr{L}) \to \Gamma(U_i,\mathscr{O}_M)$, the local equation of $s(M)$ inside $L$ is given by $t_i-f_i = 0$. So the conormal vectors defined by this equation are given by $dt_i - df_i = 0$. The equalities $s = f_ie_i = f_ig_{ij}e_j = f_je_j$ implies that $f_ig_{ij} = f_j$, and we have
\begin{align*}
\alpha_{ji}^*(dt_j - df_j) & = d(g_{ij}t_i) - df_j \\
                                     & = \left(t_idg_{ij} +g_{ij}dt_i - f_idg_{ij} - g_{ij}df_i)\right\vert_{t_i = f_i} \\
                                     & = g_{ij}(dt_i - df_i).
\end{align*}
This means that the forms $\{(dt_i - df_i)\otimes e_i\}$ with values in the line bundle $L$ glue together, defining a non-vanishing global section of $(T^*L \vert_{s(M)}) \otimes \pi^*\mathscr{L}$ (rather than $T^*L \vert_{s(M)}$)! Though one does not have a conormal vector at each point of $s(M)$, the conormal direction at each point of $s(M)$ is still well defined (by the line generated by the non-vanishing form $dt_i-df_i$). 

\begin{remark}
One idea to find an analogue of the graph deformation is the following. We fix the section $s$ and the vector bundle $(T^*L \vert_{s(M)}) \otimes \pi^*\mathscr{L}$. It is tempting to consider the locally defined forms $\{(dt_i - tdf_i) \otimes e_i\}$ for an arbitrary parameter $t \in \mathbb{C}$, and parallel to our construction in \S \ref{deformation}, stipulate that they are the direction vectors of the deformed graph $\Gamma_t$, if there were any. But one will quickly see that this idea fails because the forms constructed in this way don't glue unless $t=1$.

Another idea goes as follows. This time, we view the forms $\{(dt_i - tdf_i) \otimes e_i\}$ as defining a global section of $(T^*L \vert_{ts(M)}) \otimes \pi^*\mathscr{L}$  for an arbitrary $t \in \Cbb$. Since $(ts)^*(T^*L \otimes \pi^*\mathscr{L}) \cong (T^*M \oplus \mathscr{L}^{\vee})\otimes \mathscr{L}$, the global section defined by $\{(dt_i - tdf_i) \otimes e_i\}$ is pulled back to a global section of $(T^*M \oplus \mathscr{L}^{\vee})\otimes \mathscr{L}$. Therefore, it is tempting to compare the global sections of $(T^*M \oplus \mathscr{L}^{\vee})\otimes \mathscr{L}$ thus defined for various $t$. However, this idea still won't help us because one can quickly check that the global sections induced from different sections $ts$ are all identical. In fact, they are equal to $1 \in \Gamma(M,\mathscr{O}_M) \subset \Gamma(M, (T^*M \oplus \mathscr{L}^{\vee})\otimes \mathscr{L})$.
\end{remark}

Because $f_ie_i$ is the local expression of $s$, we see that $\{\frac{dt_i}{f_i} - t\frac{df_i}{f_i}\}$ glue to a global meromorphic section of $T^*L \vert_{ts(M)}$. Because the value of the invariant form $t\frac{dt_i}{t_i}$ is $\frac{dt_i}{f_i}$ when $t_i = tf_i$, we see that the induced meromorphic global invariant form is $\{t\frac{dt_i}{t_i} - t\frac{df_i}{f_i}\}$. Recall that $X$ is the complex subspace of $M$ defined by the zeroes of the section $s$. So we have $\{t\frac{dt_i}{t_i} - t\frac{df_i}{f_i}\} \in \Gamma(M,Q^{\vee}(X))$. We have constructed the explicit isomorphism $Q^{\vee} \xrightarrow{\sim} \mathcal{P}^1(\mathscr{L})\otimes \mathscr{L}^{\vee}$, sending $(\frac{dt_i}{t_i},w)$ to $(1,-w)$, using the local splittings $\mathcal{P}^1(\mathscr{L})\otimes \mathscr{L}^{\vee}\vert_{U_i} \cong \mathscr{O}_{U_i} \oplus \Omega^1_{U_i}$. The forms $\{t\frac{dt_i}{t_i} - t\frac{df_i}{f_i}\}$ are sent to $\{t + t\frac{df_i}{f_i}\}$. Finally, we can use $s=\{f_ie_i\} \in \Gamma(M, \mathscr{L})$ to untwist $\mathscr{L}^{\vee}$ and clear denominators. We get $\{tf_i + tdf_i\} \in \Gamma(M, \mathcal{P}^1(\mathscr{L}))$. Note that the final result is $i_r(ts)$, the truncated Taylor expansion of $ts$. 

For any $s \in \Gamma(M,\mathscr{L})$, the procedure we have described actually defines a morphism
\begin{equation*}
T^*L\vert_{s(M)} \otimes \pi^*\mathscr{L} \to \mathcal{P}^1(\mathscr{L}).
\end{equation*}
The induced map on global sections takes $\{(dt_i - df_i) \otimes e_i\}$ to $i_r(s)$. With the help of this morphism, we can ``convert" the Lagrangian specialisation to a graph deformation. Indeed, the conormal space for $ts$ is brought to the section $ti_r(s) \in \Gamma(M,\mathcal{P}^1(\mathscr{L}))$, and the Lagrangian limit is converted to the limit of the sections $ti_r(s)$ when $t \to \infty$. The latter limit clearly can be understood by the graph construction. Just apply the standard procedure of the graph construction to the morphism $\mathscr{O}_M \to \mathcal{P}^1(\mathscr{L})$ determined by $i_r(s)$. The limit cycle is given by 
\begin{equation*}
[\textup{Bl}_{Y'}M] + [\mathbb{P}(C_{Y'}M \oplus 1)],
\end{equation*}
where $Y$ is the zero of the section $i_r(s)$. In $U_i$, $Y$ is defined by $f_i$ and all partial derivatives of $f_i$. To relate $[\textup{Bl}_{Y'}M]$ and $[\mathbb{P}(C_{Y'}M \oplus 1)]$ to the constructible functions $\chi$ and $\mu$, we run almost the same arguments as in \S \ref{microlocal}, and we omit them entirely.

\section{complete intersections}\label{complete intersections}

The moral we have acquired in the previous section is that, there are two parts in the limit cycle of the graph construction; the one which dominates $M$ gives the Euler characteristic of the Milnor fiber up to a sign, and the one which is mapped into the critical space $C(f)$ gives the Milnor number up to a sign. We will discuss this statement precisely. 

\subsection{}\label{basicsetting}
Let us start with a holomorphic map $f: M \to N$ between two complex manifolds. Because we are chiefly interested in complete intersection singularities, we assume $f$ is flat, though the deformation construction which will be carried out next does not require this assumption. Flatness implies that $f$ is open, therefore Sard's theorem implies that the critical locus $|C(f)|$ (the discriminant locus $|D(f)|$) is nowhere dense in $M$ ($N$). We still consider the family of graphs $\Gamma_t \subset T^*M \oplus f^*T^*N$ ($t \neq \infty$). Unlike the $\dim N =1$ case, there are two possible versions of the graph construction for us to choose. Namely, we can form $G = \textup{Grass}_n(T^*M \oplus f^*T^*N)$ and get a family $M \times \Cbb \cong \mathscr{M} \subset G \times \Cbb$ of embeddings of $M$ into $G$ by the formula
\begin{align*}
M \times \Cbb &\to G \times \Cbb \\
(z,t) &\mapsto (\Gamma_t(z), t),
\end{align*}
or we can form $\Pbf(T^*M \oplus f^*T^*N)$ and consider $\Pbf(\Gamma_t)$. If $N = \Cbb^n$ and we consider the embedding $i_t: M \to M \times \Cbb^n$ given by $z \mapsto (z,-tf(z))$, then $\Pbf(\Gamma_t)$ is isomorphic to the projectivised conormal space of $i_t(M)$ in $M \times \Cbb^n$, as we explained in \S \ref{deformation}. We will loosely call the first choice relative Nash construction and the second choice relative conormal construction. We fix the following notations in the rest of the paper:
\begin{itemize}
\item The tautological subbundle of $G$ is denoted by $S$. Note that $S\vert_{\mathscr{M}_t} = \Gamma_t$. Let $\pr_1: G \times \Pbf^1 \to G$ is the first projection. So with our notations, $\pr^*_1S$ is a bundle over $G \times \Cbb$ and $\Gamma$ is a bundle over $M \times \Cbb$, and we have $\pr^*_1S\vert_{\mathscr{M}} \cong \Gamma$ under the isomorphism $\mathscr{M} \cong M \times \Cbb$. The tautological quotient bundle of $G$ is denoted by $Q$, and let $Q_t = \pr^*_1Q\vert_{\mathscr{M}_t}, S_t = \pr^*_1S\vert_{\mathscr{M}_t}$.
\item We still denote the tautological quotient bundle of $\Pbf(T^*M \oplus f^*T^*N)$ by $\zeta$ and view $\zeta$ as a family of bundles $\zeta_t$ on $\Pbf(\Gamma_t)$. Note that $\zeta_t$ is {\em{not}} the tautological quotient bundle of $\Pbf(\Gamma_t)$.
\item Let $\pi: \Pbf(T^*M \oplus f^*T^*N) \to M$, $\pi_M: \Pbf(T^*M) \to M$ and $p: G \to M$ be the projections.
\item Let $j: \{\infty\} \to \Pbf^1$ be the inclusion.
\end{itemize}

We have the following basic relations:
\begin{lemma}
\begin{equation*}
c(Q^{\vee}_t) \cap [\mathscr{M}_t] = c(TM) \cap [M] = (-1)^{n-1} \pi_*\Big(c(\zeta_t^{\vee}) \cap [\Pbf(\Gamma_t)]\Big) 
\end{equation*}
in $H_*(M)$.
\end{lemma}

\begin{proof}
The first equation follows from the Whitney sum formula and the fact that $S_t \cong \Gamma_t \cong f^*T^*N$ as vector bundles on M.

In the case that $N = \Cbb^n$, we have $T^*M \oplus f^*T^*N \cong \Omega^1_{M \times \Cbb^n}\vert_{M_t}$, and the second equation follows immediately from Lemma 1 of \cite{MR1063344}. (Setting m = m+n, d= m and $\textup{taut} = \zeta^{\vee}$ loc. cit.) In the general case, it follows from easy manipulation of properties of Chern classes of vector bundles. We omit the details.

\end{proof}

Similarly, we have the corresponding statements for constructible functions.
\begin{lemma}\label{nash-conor}
\begin{equation*}
1 = 1_{M}(z) = \int c(Q^{\vee}_t) \cap s(p^{-1}(z)\cap \mathscr{M}_t,\mathscr{M}_t) = \int (-1)^{n-1} c(\zeta_t^{\vee}) \cap s\Big(\pi^{-1}(z) \cap \Pbf(\Gamma_t),\Pbf(\Gamma_t)\Big) 
\end{equation*}
for any $z \in M$.
\end{lemma}

\begin{proof}
Omitted.
\end{proof}

\subsection{}\label{nash}
As is explained in \cite{MR732620} \S 18.1, we can group the limiting cycle $[\mathscr{M}_\infty]$ into two parts. One of them dominates $M$ and the other is mapped into $C(f)$, the locus of points where $0 \to f^*T^*N \to T^*M$ fails to be left exact. We observe that
\begin{proposition}
The component of $\mathscr{M}_\infty$ which dominates $M$ is $\tilde{M} = \textup{Bl}_{C(f)}M$, and $\tilde{M}$ can be identified with the Nash modification of $M$ relative to $f$. 
\end{proposition} 

For the fluency of the flow of the discussion, we choose to formally define $\mathscr{M}_\infty$ and  $[\mathscr{M}_\infty]$ in \S \ref{conor}. The reader can rely on the intuition that $\mathscr{M}_\infty$ is the deformation limit of $\mathscr{M}_t$ for the moment. Recall also the construction of the relative Nash modification. In our context, because $f$ is a submersion at every point $x \in M \setminus C(f)$, we have a section of $\textup{Grass}_{m-n}(TM)$ over $M \setminus C(f)$, sending each $z \in M \setminus C(f)$ to $T_zM_{f(z)}$. The Nash modification relative to $f$ is the closure of the image of $M \setminus C(f)$ in $\textup{Grass}_{m-n}(TM)$.

\begin{proof}
The morphism $df:f^*T^*N \to T^*M$ induces $\Lambda^n(df):f^*\Lambda^n T^*N \to \Lambda^n T^*M$. Twisting the dual of $\Lambda^n(df)$ by $f^* \Lambda^n T^*N$, we get $ \Lambda^n TM \otimes f^* \Lambda^n T^*N \to \Osr_{M}$. Denote by $\mathscr{J}$ the image of this morphism, which is also the ideal sheaf of $C(f)$. Therefore we have a surjective morphism $\Lambda^n TM \to \mathscr{J} \otimes \Lambda^n TN$, which allows us to construct the surjective morphism of algebras:
\begin{equation*}
\textup{Sym}(\Lambda^n TM) \to \textup{Rees}(\mathscr{J} \otimes \Lambda^n TN). 
\end{equation*}

The statements of the proposition then follows from the observations below. 
\begin{itemize}
\item $\textup{Proj}\Big(\textup{Sym}(\Lambda^n TM)\Big) \cong \Pbf(\Lambda^n T^*M)$.
\item $\textup{Proj}\Big(\textup{Rees}(\mathscr{J} \otimes \Lambda^n TN)\Big) \cong \textup{Proj}\Big(\textup{Rees}(\mathscr{J})\Big)$. Both spaces are isomorphic to $\textup{Bl}_{C(f)}M$. The difference between these two $\textup{Proj}\Big(\textup{Rees}(\ldots)\Big)$ is that, they have difference tautological line bundles $\Osr(1)$. We have $\Osr_1(1) \cong \phi^*\Osr_2(1) \otimes \Lambda^n TN$ where $\phi$ is the isomorphism from the first Proj to the second one, and the subscripts $1,2$ indicates the first or second Proj.
\item The closed embedding $\textup{Bl}_{C(f)}M \to \Pbf(\Omega^n_M)$ factors through the Pl{\"{u}}cker embedding $\textup{Grass}_n(T^*M) \to  \Pbf(\Lambda^n T^*M)$. Locally on $M$ and $N$, when we can expression $f$ by coordinates $f=(f_1,\ldots,f_n)$, the map $(M \setminus C(f)) \to \Pbf(\Lambda^n T^*M)$ is given by $z \mapsto \Cbb df_1(z)\wedge\ldots\wedge df_n(z)$. Therefore we have the inclusion $M \setminus C(f) \subset \textup{Grass}_n(T^*M)$. Since $C(f)$ is nowhere dense in $M$ by our flatness assumption, $\textup{Bl}_{C(f)}M$ contains a dense open subset isomorphic to $M\setminus C(f)$. This shows that $\textup{Bl}_{C(f)}M = \overline{M\setminus C(f)} \subset \textup{Grass}_n(T^*M)$. 
\item Fixing a point $z \in (M \setminus C(f))$. The deformed graph $\Gamma_t$ at $z$ corresponds to the point $\Cbb(dw_1 + tdf_1(z))\wedge \ldots \wedge (dw_n + tdf_n(z))$ on $G \subset \Pbf\Big(\Lambda^n(T^*M\oplus f^*T^*N)\Big)$. Here $w_1,\ldots,w_n$ are local coordinates on $N$. It is clear that when $t \to \infty$, the limit point is $\Cbb df_1(z)\wedge\ldots\wedge df_n(z)$. This shows that $M\setminus C(f)$ appears in the deformation limit, hence also its closure $\textup{Bl}_{C(f)}M$.
\item Under the isomorphism $\textup{Grass}_n(T^*M) \cong \textup{Grass}_{m-n}(TM)$, the subspace represented by $df_1(z)\wedge \ldots \wedge df_n(z)$ is sent to $T_zM_{f(z)}$ for $z\in M\setminus C(f)$.
\end{itemize}

\end{proof}

Having identified $\tilde{M}$ as an analytic subspace of $\mathscr{M}_\infty$, we can now let $\widetilde{C(f)} = \overline{\mathscr{M}_\infty \setminus \tilde{M}}$, and we have 
\begin{equation}\label{deflimit}
[\mathscr{M}_\infty] = [\tilde{M}] + [\widetilde{C(f)}].
\end{equation}

Picking $z \in C(f)$, and using the deformation argument, we have
\begin{align*}
1 &=  \int c(Q^{\vee}_0) \cap s(p^{-1}(z)\cap \mathscr{M}_0, \mathscr{M}_0) \\
   &=  \int c(Q^{\vee}_\infty) \cap s(p^{-1}(z)\cap \mathscr{M}_{\infty}, \mathscr{M}_{\infty}) \\
   &= \int c(Q^{\vee}_\infty) \cap s(p^{-1}(z) \cap \tilde{M}, \tilde{M}) + \int c(Q^{\vee}_\infty) \cap s(p^{-1}(z) \cap \widetilde{C(f)}, \widetilde{C(f)}) \\
   &= \int c(p^*_{\tilde{M}}f^*TN)c(Q^{\vee}_M)\cap s(p^{-1}_{\tilde{M}}(z), \tilde{M}) + \int c(Q^{\vee}_\infty) \cap s(p^{-1}(z) \cap \widetilde{C(f)}, \widetilde{C(f)}),
\end{align*}
where $Q_M$ denotes the tautological quotient bundle of $\textup{Grass}_n(T^*M)$ and $p_{\tilde{M}}$ denotes the restriction $p\vert_{\tilde{M}}$. The last equation uses the fact the $\tilde{M}$ is contained entirely in $\textup{Grass}_n(T^*M)$ so that $(Q_{\infty})\vert_{\tilde{M}} \cong Q_M\vert_{\tilde{M}} \oplus p^*_{\tilde{M}}f^*\Omega_N$. Note that $Q_M^{\vee}\vert_{\tilde{M}}$ can be identified with the rank $m-n$ relative Nash tangent bundle $\tilde{T}_{M/N}$, and that $p^*_{\tilde{M}}f^*TN$ is trivial over $p^{-1}_{\tilde{M}}(z)$.

Compare this result with the equation which defines the Milnor number
\begin{equation*}
1 = \chi(z) + (-1)^{m-n+1}\mu(z),
\end{equation*}
we claim that
\begin{theorem}\label{nash-integral}
Let $f: M \to N$ be a holomorphic map between two complex manifolds and let $z \in C(f)$. Assume that in a neighbourhood of $z$, $f$ is flat and sans \'{e}clatement en codimension $0$, then we have the following formulas for the Euler characteristic of the Milnor fibre at $z$ and the Milnor number at $z$: 
\begin{align*}
\chi(z) &= \int c(\tilde{T}_{M/N})\cap s(p^{-1}_{\tilde{M}}(z), \tilde{M}), \\           
\mu(z) & = (-1)^{m-n+1} \int c(Q^{\vee}_\infty) \cap s(p^{-1}(z) \cap \widetilde{C(f)}, \widetilde{C(f)}).    
\end{align*}
If moreover $\tilde{M}$ is Cohen-Macauley, or $n=1$, we also have
\begin{equation*}
\chi(z) = \int c(\tilde{T}_{M/N})\cap s\Big(p^{-1}_{\tilde{M}}(z), p^{-1}_{\tilde{M}}M_{f(z)}\Big).
\end{equation*}
\end{theorem}

\subsection{}\label{noblowup}
Before we enter into the proof of Theorem \ref{nash-integral}, we need to review the notion ``sans \'{e}clatement en codimension $0$" and various consequences of this condition. 

Let $M,N$ be two complex manifolds and let $Z$ be a reduced and irreducible analytic subspace of $M$. For any morphism $g: Z \to N$, we may consider its factorisation through the graph

\begin{equation*}
\begin{tikzcd}
Z \arrow[rd,"g"] \arrow[r] & M\times N \arrow[d] \\
                                  & N, 
\end{tikzcd}
\end{equation*}
so that $Z$ can be considered as a family of subspaces of $M$ parametrised by $N$. Suppose there is a dense open subset $Z^{\circ} \subset Z$ such that $g\vert_{Z^{\circ}}$ is a submersion; the relative conormal space $T^*_{g}M$ is the closure of the set $\{(z,\alpha)\in T^*M \vert_{Z^{\circ}}; \ \alpha(T_zZ_{g(z)}) =0\}$ in $T^*M$. Denote by $\tau_g$ the projection $T^*_gM \to Z$.

\begin{definition}\cite{MR760677}
The morphism $g$ is san \'{e}clatement en codimension $0$ if all the fibres of the composition map $g \circ \tau_g$ have the same dimension $m$.
\end{definition}

This condition has the following important consequences.

\begin{theorem}\label{Thomag}\cite{MR760677}
If $g$ is sans \'{e}clatement en codimension $0$, and if $W$ is a closed analytic subset of Z, there exists an open Zariski dense subset $W^{\circ}$ of $W$ such that the pair $(Z^{\circ},W^{\circ})$ satisfies the Thom $A_g$ condition. 
\end{theorem}

\begin{remark}
The existence of a stratification in $Z$ such that Thom $A_g$ condition is satisfied between any pair of strata guarantees the existence of the Milnor fibration. See \cite{MR3767427} for detailed explanations in both real and complex analytic cases.
\end{remark}

\begin{theorem}\label{Lagrangianlimit}\cite{MR760677}
Let $g: Z \to N$ be sans \'{e}clatement en codimension $0$, and let $h: S \to N$ be any analytic morphism. Also let $\Gamma$ be an irreducible component of $T_g^{*}M \times _N S$. Then $\Gamma$ can be identified with the relative conormal space of the morphism $\tau^{'}_g(\Gamma) \to N$. Here $\tau^{'}_g: T^*_gM \times_N S \to Z\times_N S$ is the base change of the morphism $\tau_g$.
\end{theorem}

Given a point $t \in \overline{g(Z^{\circ})}$, we can view the fibre $\tau^{-1}_gg^{-1}(t)$ as the limit of the family of conormal spaces associated to $g: Z \to N$. Theorem \ref{Lagrangianlimit} implies that each irreducible component of $\tau^{-1}_gg^{-1}(t)$ is the conormal space of some irreducible analytic subset of $M$. Therefore the cycle $[\tau^{-1}_gg^{-1}(t)]$ is a conic Lagrangian cycle . In general, one cannot expect $\tau^{-1}_gg^{-1}(t) = T^*_{g^{-1}(t)}M$ even for those $t \in g(Z^{\circ})$.

Let's apply these considerations on the basic setting of \S\ref{basicsetting}. We have a flat morphism $f: M \to N$ between two complex manifolds. Consider the graph embedding $M \to M \times N$ so that $M$ plays the role of $Z$ in the preceding paragraphs.

\begin{proposition}\label{nbl}
In either of the following cases
\begin{enumerate}
\item $\dim N =1$;
\item $f: M \to N$ is of finite contact type,
\end{enumerate}
$f$ is sans \'{e}clatement en codimension $0$.
\end{proposition}

\begin{proof}
In case (1), $\dim N = 1$ implies that $f \circ \tau_{f}: T^*_fM \to N$ is flat. Hence all the fibres have the same dimension. In case (2), since $\dim T^*_fM = m+n$, we know $\dim \tau^{-1}_f f^{-1}(t) \geq m$ for any $t \in N$. On the other hand, since $\tau^{-1}_f f^{-1}(t)$ has only isolated singularities, $\tau^{-1}_f f^{-1}(t)$ has one irreducible component $T^*_{f^{-1}(t)}M$ dominating $f^{-1}(t)$, and one extra irreducible component for each singular point of $f^{-1}(t)$. If $z$ is one of the singular points, the irreducible component of $\tau^{-1}_f f^{-1}(t)$ mapping down to $z$ is a cone in $T^*_zM$, hence its dimension is at most $m$. 
\end{proof}

Unless otherwise stated, we will always assume $f$ is flat and sans \'{e}clatement en codimension $0$. Fix an arbitrary point $z \in M$, let $\Gamma_1, \ldots, \Gamma_r$ be the irreducible components of $\tau^{-1}_f(M_{f(z)})$, and let $Z_i = \tau_f(\Gamma_i)$. By our previous discussion, $\Gamma_i = T^*_{Z_i}M$ for each $i$. We wish to understand the irreducible components of $p^{-1}_{\tilde{M}}(M_{f(z)})$ in terms of $Z_1, \ldots, Z_r$. In fact, one can see that the correspondence between the conormal space and the Nash transformation (discussed for example in \cite{MR1063344} \S 1) extends  to a correspondence between the relative conormal space and the relative Nash transformation. We have the following commutative diagram
\begin{equation*}
\begin{tikzcd}
S\vert_{\tilde{M}} \arrow[r] \arrow[d] & T^*_fM \arrow[d,"\tau_f"] \\
\tilde{M} \arrow[r,"p_{\tilde{M}}"]                              & M 
\end{tikzcd}
\end{equation*}
where the horizontal arrows are proper and bimeromorphic. Restricting this diagram to $M_{f(z)}$, we get a proper surjective morphism from $S\vert_{p^{-1}_{\tilde{M}}(M_{f(z)})}$ to $\tau^{-1}_f(M_{f(z)})$. Let $W$ be an irreducible component of $p^{-1}_{\tilde{M}}(M_{f(z)})$ and let $Z' = p_{\tilde{M}}(W)$. By Theorem \ref{Thomag}, $Z'$ contains an open regular subset $Z^{\circ}$ such that $W^{\circ} := p^{-1}_{\tilde{M}}(Z^{\circ}) \cap W$ is open in $K$ and $S\vert_{W^{\circ}} \subset T^*_{Z^{\circ}_i}M$. Taking the closure, we conclude that $S\vert_{W}$ is mapped bimeromorphically onto an irreducible subset of $T^*_{Z'}M$. Therefore, $\dim S\vert_{W} \leq m$ and $\dim W \leq m-n$. On the other hand, $p^{-1}_{\tilde{M}}(M_{f(z)})$ is the fibre of $f\circ p_{\tilde{M}}$ over $f(z)$. The dimensions of its irreducible components are no less than $\dim \tilde{M} - \dim N = m-n$. This argument shows that $\dim W = m-n$, and $Z'$ is one of the $Z_1, \ldots, Z_r$, and $S\vert_W$ is mapped bimeromorphically onto $T^*_{Z'}M$. This proves the following

\begin{proposition}\label{equidim}
$p^{-1}_{\tilde{M}}(M_{f(z)})$ is purely $(m-n)$-dimensional. Its irreducible components are in 1-1 correspondence with the irreducible components of $\tau^{-1}_f(M_{f(z)})$. 
\end{proposition}

\begin{remark}
Given a morphism $f:M \to \mathbb{C}^n$, there is one more Lagrangian specialisation to consider. In the diagram
 
 \begin{equation*}
\begin{tikzcd}
M \times \mathbb{C} \arrow[r] \arrow[rd]  & (M \times \mathbb{C}^n) \times \mathbb{P}^1 \arrow[d] \\
                                                                 & \mathbb{P}^1 
\end{tikzcd}
\end{equation*}
where the horizontal arrow is given by $(z,t) \to (z,-tf(z),t)$, we can view $M \times \mathbb{C}$ as a family of graphs in $M \times \mathbb{C}^n$. The basic consideration of this paper is to study the specialisation of the family of conormal spaces associated with this family of graphs. In the general situation where we start with a morphism $f: M \to N$, the family of graphs in $M \times N$ cannot be defined, but the graph construction can serve as a substitute for the deformation of conormal spaces.
 
\end{remark}

\subsection{}
We can move on to prove Theorem \ref{nash-integral} now. The proof is a rather straightforward generalisation of the theory established in \cite{MR629121} and \cite{MR804052}. 

Take any $z \in M$. We will only concern the local structure of $f$ near $z$ so that without loss of generality we assume $M$ ($N$) is an open subset of $\mathbb{C}^m$ ($\mathbb{C}^n$). Choose holomorphic coordinate systems on $M$ and $N$ such that $z = 0 \in \mathbb{C}^m$ and $f(0)=0 \in \mathbb{C}^n$. We also fix an Hermitian metric $H$ on $\mathbb{C}^m$. Let $B_{\epsilon}$ ($D^{\circ}_{\eta}$) be a small $\epsilon$-ball (open $\eta$-ball) around $z$ ($f(z)$), with $0<\eta \ll \epsilon \ll 1$, and let $B^{\circ}_{\epsilon} = B_{\epsilon} \setminus \partial B_{\epsilon}$. There is a radial section $r$ of the real tangent bundle of $M$, sending each $z' \in M$ to $r(z') = \overrightarrow{zz'}$. If we use the canonical $\mathbb{R}$-linear isomorphism between the real tangent bundle of $M$ and the real vector bundle underlying $TM$, $r$ corresponds to the holomorphic section $2\Sigma z_i\frac{\partial}{\partial z_i}$. Let $\tilde{r}$ be the section of $p^{*}_{\tilde{M}}TM$ defined by $p^{*}_{\tilde{M}}(r)$. Over $\tilde{M}$ we have the exact sequence

\begin{equation*}
0 \to \tilde{T}_{M/N} \to p^{*}_{\tilde{M}}TM \to S^{\vee}\vert_{\tilde{M}} \to 0,
\end{equation*}
and with the help of the Hermitian metric $H$, we get a (non-holomorphic) splitting
\begin{equation*}
p^{*}_{\tilde{M}}TM \cong \tilde{T}_{M/N} \oplus S^{\vee}\vert_{\tilde{M}}.
\end{equation*} 
We denote the component of $\tilde{r}$ in $\tilde{T}_{M/N}$ by $\sigma_{H}$. It is just a $C^{\infty}$ section of $\tilde{T}_{M/N}$. 

Given a point $\delta \in N$, we let $\delta_{\tilde{T}_{M/N}}: \tilde{T}_{M/N} \to \tilde{T}_{M/N} \times N$ be the closed embedding whose first factor is the identity and the second factor is the constant map to $\delta$, i.e. $(id_{\tilde{T}_{M/N}}\text{,}\delta)$. Note that the normal bundle of $\delta_{\tilde{T}_{M/N}}$ is trivial. Setting things up in this way, the following Cartesian diagram (in the category of sets) will play an important role in the sequel.

\begin{equation}\label{intersect0section}
\begin{tikzcd}
p^{-1}_{\tilde{M}}(M_{f(0)}) \arrow[r] \arrow[d] & \tilde{M} \arrow[d,"(\sigma_{H}\text{,} f \circ p_{\tilde{M}})"] \\
\tilde{T}_{M/N}  \arrow[r,"0_{\tilde{T}_{M/N}}"]                   & \tilde{T}_{M/N} \times N 
\end{tikzcd}
\end{equation}

Let $K = p^{-1}_{\tilde{M}}((B_{\epsilon} \times D^{\circ}_{\eta}) \cap M)$ and $\partial K = p^{-1}_{\tilde{M}}((\partial B_{\epsilon} \times D^{\circ}_{\eta}) \cap M)$ where $M$ is identified with the graph of $f$ and the intersection takes place inside $M \times N$. We also let $K^{\circ} = K \setminus \partial K$. The figure in \cite{MR747303} \S 2.B may help the reader to see what is going on.

\begin{proposition}
$\sigma_H$ is non-vanishing on $\partial K$. Because this holds true for any sufficiently small $\epsilon$, $\sigma_H$ is in particular non-vanishing on $p^{-1}_{\tilde{M}}(B_{\epsilon} \cap M_{f(0)}\setminus \{0\})$.
\end{proposition}
\begin{proof}
The proof follows word by word from the proof of a corresponding proposition in the absolute case ($N$ is a point) given by Gonzalez-Sprinberg (\cite{MR629121} \S 4.1, page 12), except that we replace the use of Nash modification by the relative Nash modification $\tilde{M}$ and the use of Whitney $A$ condition by Thom $A_f$ condition whose validity is inferred from Theorem \ref{Thomag}.
\end{proof}

This proposition indicates that we can extend diagram \eqref{intersect0section} into the following diagram
\begin{equation*}
\begin{tikzcd}
p^{-1}_{\tilde{M}}(0) \arrow[r] \arrow[d] & K^{\circ} \cap p^{-1}_{\tilde{M}}(M_{f(0)}) \arrow[r] \arrow[d,"\sigma_H"] & K^{\circ} \arrow[d,"\tilde{\sigma}_{H}"] \\
\tilde{M} \arrow[r,"0_{\tilde{M}}"] & \tilde{T}_{M/N} \arrow[r,"0_{\tilde{T}_{M/N}}"] & p^{*}_1\tilde{T}_{M/N}
\end{tikzcd}
\end{equation*}
where both squares are Cartesian in the category of sets. Intuitively, we wish to identify
\begin{equation*}
\int c(\tilde{T}_{M/N})\cap s(p^{-1}_{\tilde{M}}(0), \tilde{M}) = \int c(\tilde{T}_{M/N})\cap s(p^{-1}_{\tilde{M}}(0), K^{\circ})
\end{equation*}
with the degree of $0^{!}_{\tilde{M}}0^{!}_{\tilde{T}_{M/N}}[K^{\circ}] \in H^{BM}_{0}(p^{-1}_{\tilde{M}}(0))$. The problem, of course, is that $\sigma_H$ is not holomorphic, and consequently these Gysin maps are not well defined. To overcome this difficulty, we let $\nu: \textup{Grass}_n(p^{*}_{\tilde{M}}TM) \to \tilde{M}$ be the projection and let $U$ be the open subset in $\textup{Grass}_n(p^{*}_{\tilde{M}}TM)$ over which the tautological subbundle of $\nu^{*}p^{*}_{\tilde{M}}TM$ is a complement of $\nu^*\tilde{T}_{M/N}$. Namely, 
\begin{equation*}
\nu^{*}p^{*}_{\tilde{M}}TM \cong \nu^*\tilde{T}_{M/N} \oplus \textup{taut}
\end{equation*}
over $U \subset \textup{Grass}_n(p^{*}_{\tilde{M}}TM)$. Note that $U$ is a principle homogeneous space with the Abelian group $\textup{Hom} (\mathbb{C}^n, \mathbb{C}^{m-n})$. Therefore $\nu \vert_{U}$ is flat. We also let $\hat{K} = (\nu^{-1}K^{\circ}) \cap U$ and $\hat{\nu} = \nu \vert_{\hat{K}}$. With this construction, we can split the section $\hat{\nu}^{*}(\tilde{r})$ analytically. Let $\sigma_1$ ($\sigma_2$) be its component in $\hat{\nu}^*\tilde{T}_{M/N}$ ($\textup{taut}\vert_{\hat{K}}$). 

\begin{proposition}
Let $u \in \hat{\nu}^{-1}p^{-1}_{\tilde{M}}(0)$, and let $c>0$ be a constant. There exists a neighbourhood $W_u$ of $u$ in $\hat{\nu}^{-1}p^{-1}_{\tilde{M}}(M_{f(0)})$ such that $\|\sigma_2\| < c \|\sigma_1\|$ in $W_u$. 
\end{proposition}

\begin{proof}
Duplicate the proof of the corresponding proposition on \cite{MR629121} page 17, with the Whitney $A$ condition replaced by Thom $A_f$ condition. The figures on page 17 and 18 loc.cit. are quite enlightening.
\end{proof}

\begin{corol}
Let $Z(\sigma_1)$ be the zero space of $\sigma_1$. $\hat{\nu}^{-1}p^{-1}_{\tilde{M}}(0)$ is closed as an analytic subspace of $Z(\sigma_1) \cap \hat{\nu}^{-1}p^{-1}_{\tilde{M}}(M_{f(0)})$, and is open as a subset of $|Z(\sigma_1) \cap \hat{\nu}^{-1}p^{-1}_{\tilde{M}}(M_{f(0)})|$.
\end{corol}

\begin{proof}
Clearly $0 \in M$ is the only zero of the radial section $r$, so $\hat{\nu}^{-1}p^{-1}_{\tilde{M}}(0)$ is the zero space of $\hat{\nu}^*(\tilde{r})$. In local coordinates, the components of $\sigma_1$ is a linear combination of the components of $\hat{\nu}^*(\tilde{r})$ with holomorphic coefficients. This shows that $\hat{\nu}^{-1}p^{-1}_{\tilde{M}}(0)$ is a closed analytic subspace of $Z(\sigma_1)$. To show the openness, we just need to show that for any $u \in \hat{\nu}^{-1}p^{-1}_{\tilde{M}}(0)$, there exists an open neighbourhood $W_u$ in $\hat{\nu}^{-1}p^{-1}_{\tilde{M}}(M_{f(0)})$ such that $|\hat{\nu}^{-1}p^{-1}_{\tilde{M}}(0) \cap W_u| = |Z(\sigma_1) \cap W_u|$. Such $W_u$ is provided by the previous proposition.
\end{proof}

Let $K_{f(0)}'' = \Big(Z(\sigma_1) \cap \hat{\nu}^{-1}p^{-1}_{\tilde{M}}(M_{f(0)})\Big) \setminus \hat{\nu}^{-1}p^{-1}_{\tilde{M}}(0)$ be the union of the connected components of $Z(\sigma_1) \cap \hat{\nu}^{-1}p^{-1}_{\tilde{M}}(M_{f(0)})$ disjoint from $\hat{\nu}^{-1}p^{-1}_{\tilde{M}}(0)$, and let $K_{f(0)}' =\hat{\nu}^{-1}p^{-1}_{\tilde{M}}(M_{f(0)}) \setminus K_{f(0)}''$ be the open subspace of $\hat{\nu}^{-1}p^{-1}_{\tilde{M}}(M_{f(0)})$ where $\hat{\nu}^{*}(\tilde{r})$ and $\sigma_1$ have the same vanishing locus. Denote by $I$ and $J$ the ideals of $Z(\sigma_1) \cap K_{f(0)}'$ and $\hat{\nu}^{-1}p^{-1}_{\tilde{M}}(0)$ in $K_{f(0)}'$ respectively. We have $I \subset J$ by the previous corollary. And moreover

\begin{corol}\label{idealintegral2}
$I$ and $J$ have the same integral closure in $\mathscr{O}_{K_{f(0)}'}$, and hence (their inverse images) have the same integral closure in $\mathscr{O}_{\hat{K}\setminus K_{f(0)}''}$.
\end{corol}

\begin{proof}
This follows immediately from the second in the list of the criterions for integral dependence given in \cite{MR0568901} \S 2.7. Alternatively, one can check the proof of Corollary 2 on page 20 of \cite{MR629121}.
\end{proof}

To give a quick summary, we are now in the situation of the following diagram

\begin{equation*}
\begin{tikzcd}
Z(\sigma) \cap K_{f(0)}' \arrow[r] \arrow[d] & K_{f(0)}' \arrow[r] \arrow[d,"\sigma_1"] & \hat{K}\setminus K_{f(0)}'' \arrow[d,"(\sigma_1\text{,} f\circ p_{\tilde{M}} \circ \hat{\nu})"] \\
\hat{K} \arrow[r] \arrow[d] & \hat{\nu}^{*}\tilde{T}_{M/N} \arrow[r] \arrow[d] & \hat{\nu}^{*}\tilde{T}_{M/N} \times N \arrow[d] \\
\tilde{M} \arrow[r,"0_{\tilde{M}}"] & \tilde{T}_{M/N} \arrow[r,"0_{\tilde{T}_{M/N}}"] & \tilde{T}_{M/N} \times N
\end{tikzcd}
\end{equation*}
where all squares are Cartesian. 

Using Corollary \ref{idealintegral2}, we see that $|Z(\sigma) \cap K_{f(0)}'| = |\hat{\nu}^{-1}p^{-1}_{\tilde{M}}(0)|$ and moreover  
\begin{equation*}
s(Z(\sigma) \cap K_{f(0)}',\hat{K}\setminus K_{f(0)}'') = s(\hat{\nu}^{-1}p^{-1}_{\tilde{M}}(0), \hat{K}\setminus K_{f(0)}'') = s(\hat{\nu}^{-1}p^{-1}_{\tilde{M}}(0), \hat{K}).
\end{equation*}

We also have $[\hat{K}\setminus K_{f(0)}''] \in H^{BM}_{2(m+n(m-n))}(\hat{K}\setminus K_{f(0)}'')$, so that 

\begin{align*}
(0_{\tilde{T}_{M/N}}\circ 0_{\tilde{M}})^{!}[\hat{K}\setminus K_{f(0)}''] &= \Big\{c(\hat{\nu}^{*}\tilde{T}_{M/N}) \cap s(Z(\sigma) \cap K_{f(0)}',\hat{K}\setminus K_{f(0)}'')\Big\}_{2n(m-n)} \\
&= \Big\{c(\hat{\nu}^{*}\tilde{T}_{M/N}) \cap s(\hat{\nu}^{-1}p^{-1}_{\tilde{M}}(0),\hat{K})\Big\}_{2n(m-n)} \\
&= \Big\{c(\tilde{T}_{M/N})\cap s(p^{-1}_{\tilde{M}}(0), K^{\circ})\Big\}_{0} \\
&= \Big\{c(\tilde{T}_{M/N})\cap s(p^{-1}_{\tilde{M}}(0), \tilde{M})\Big\}_{0}
\end{align*}
where in the third equality we use the identification $H^{BM}_{\bullet}(p^{-1}_{\tilde{M}}(0)) \cong H^{BM}_{\bullet + 2n(m-n)}(\hat{\nu}^{-1}p^{-1}_{\tilde{M}}(0))$.

On the other hand, if we take $\delta \in D^{\circ}_{\eta}$ such that $M_{\delta} \cap K^{\circ}$ defines a Milnor fibre $F$ (without the boundary), then it is clear that $p_{\tilde{M}}: \tilde{M} \to M$ restricts to the identity in a neighbourhood of $F$, and $\tilde{T}_{M/N} \Big|_{F} \cong TF$, and $\delta^{!}_{\tilde{T}_{M/N}}[\hat{K}\setminus K_{f(0)}''] = [\hat{\nu}^{-1}F]$ is just the class of a principal $\textup{Hom} (\mathbb{C}^n, \mathbb{C}^{m-n})$-bundle over $F$. Therefore we have
\begin{align*}
(0_{\tilde{T}_{M/N}}\circ 0_{\tilde{M}})^{!}[\hat{K}\setminus K_{f(0)}''] &= 0^{!}_{\tilde{M}}0^{!}_{\tilde{T}_{M/N}}[\hat{K}\setminus K_{f(0)}''] \\
&= 0^{!}_{\tilde{M}}\delta^{!}_{\tilde{T}_{M/N}}[\hat{K}\setminus K_{f(0)}''] \\
&= \Big\{c(\hat{\nu}^{*}TF) \cap [\hat{\nu}^{-1}F]\Big\}_{2n(m-n)} \\
&= \Big\{c(TF)\cap [F]\Big\}_{0}.
\end{align*}

Taking the degree on both sides, we get
\begin{equation*}
\int c(\tilde{T}_{M/N})\cap s(p^{-1}_{\tilde{M}}(0), \tilde{M}) = \chi_c(F) = \chi(F).
\end{equation*}

We can also do the computation slightly differently. Note that
\begin{equation*}
0^{!}_{\tilde{T}_{M/N}}[\hat{K}\setminus K_{f(0)}''] = \Big\{s(K_{f(0)}', \hat{K}\setminus K_{f(0)}'')\Big\}_{2((n+1)(m-n))}.
\end{equation*}
By definition, $K_{f(0)}'$ is open in $\hat{\nu}^{-1}p^{-1}_{\tilde{M}}(M_{f(0)})$, and the latter space is purely $(n+1)(m-n)$-dimensional by proposition \ref{equidim}. Therefore $K_{f(0)}'$ is also purely $(n+1)(m-n)$-dimensional. Its every irreducible component is a proper component for the Gysin map $0^{!}_{\tilde{T}_{M/N}}$ in the sense of \cite{MR732620} Definition 7.1. However, one usually doesn't have $[K_{f(0)}'] = 0^{!}_{\tilde{T}_{M/N}}[\hat{K}\setminus K_{f(0)}'']$. This happens when the defining ideal of $K_{f(0)}'$, i.e. $(f_1,\ldots, f_n)$, is a regular sequence in $\mathscr{O}_{\hat{K}\setminus K_{f(0)}''}$. When $n=1$, this is always true, as is implied by the flatness assumption. When $n>1$, a sufficient condition for this to be true is that $\tilde{M} = \textup{Bl}_{C(f)}M$ is Cohen-Macauley. 

\begin{question}
Let $f$ be flat and sans \'{e}clatement en codimension $0$. Can one expect $\tilde{M} = \textup{Bl}_{C(f)}M$ to be Cohen-Macauley?
\end{question}
After a quick search of relevant results from literation, we decide that it is at least not a completely trivial question to ask. In the special case that $f$ has finite contact type, a famous result states that $C(f)$ is Cohen-Macauley (\cite{MR747303} proposition 4.4). So it seems not unreasonable to ask this question. 

Assuming $[K_{f(0)}'] = 0^{!}_{\tilde{T}_{M/N}}[\hat{K}\setminus K_{f(0)}'']$, we have
\begin{align*}
(0_{\tilde{T}_{M/N}}\circ 0_{\tilde{M}})^{!}[\hat{K}\setminus K_{f(0)}''] &= 0^{!}_{\tilde{M}}[K_{f(0)}']\\
&= \Big\{c(\hat{\nu}^{*}\tilde{T}_{M/N})\cap s(Z(\sigma) \cap K_{f(0)}', K_{f(0)}')\Big\}_{2n(m-n)} \\
&= \Big\{c(\hat{\nu}^{*}\tilde{T}_{M/N})\cap s(\hat{\nu}^{-1}p^{-1}_{\tilde{M}}(0), K_{f(0)}')\Big\}_{2n(m-n)} \\
&= \Big\{c(\hat{\nu}^{*}\tilde{T}_{M/N})\cap s(\hat{\nu}^{-1}p^{-1}_{\tilde{M}}(0), \hat{\nu}^{-1}p^{-1}_{\tilde{M}}(M_{f(0)})\Big\}_{2n(m-n)} \\
&=  \Big\{c(\tilde{T}_{M/N})\cap s(p^{-1}_{\tilde{M}}(0), p^{-1}_{\tilde{M}}(M_{f(0)}))\Big\}_{0}
\end{align*}
where we need to use Corollary \ref{idealintegral2} again in deriving the third equality.

With a slight shift of perspective, we can relate our computation to the local Euler obstruction. Indeed, $[K^{\circ}] \in H^{BM}_{2m}(K^{\circ})$ can be viewed as a fundamental class relative to $f$. The vector bundle $\tilde{T}_{M/N}$ has a Thom cohomology class $\omega \in H^{2(m-n)}(\tilde{T}_{M/N},\tilde{T}_{M/N} \setminus \tilde{M})$. Our $\sigma_H$ induces a map $(K, \partial K) \to (\tilde{T}_{M/N}, \tilde{T}_{M/N} \setminus \tilde{M})$. Pulling $\omega$ back, we get an obstruction class $\sigma^{*}_H(\omega) \in H^{2(m-n)}(K, \partial K) \cong H^{2(m-n)}(K^{\circ})$. The relative local Euler obstruction $\textup{Eu}_f(0)$ is $\sigma^{*}_H(\omega) \cap [K^{\circ}] \in H^{BM}_{2n}(K^{\circ})$. For any $\delta \in D^{\circ}_{\eta}$, we can pull back $\textup{Eu}_f(0)$ via the Gysin map associated with the inclusion $i_\delta: \delta \to D^{\circ}_{\eta}$, and we obtain $i^{!}_{\delta}(\textup{Eu}_{f}(0)) \in H^{BM}_0(p^{-1}_{\tilde{M}}(M_{\delta})\cap K^{\circ})$. Its degree is independent of the chosen $\delta$ and is nothing but $\chi(F)$. Roughly speaking, $\omega$ is the cohomology class in $H^{2(m-n)}(\tilde{T}_{M/N})$ dual to the zero section $\tilde{M}$, and pulling $\omega$ back via $\sigma_H$ is equivalent to intersecting $\sigma_H(K^{\circ})$ with the zero section $\tilde{M}$. These words make good sense in topology, but to make sense the intersection in algebraic geometry, we are forced to use the auxiliary construction $\hat{K}$, which is technically essential but conceptually less significant.

\subsection{}\label{conor}
We have finished the discussion on the Nash version of the graph construction. With the help of Lemma \ref{nash-conor}, we can quickly derive the corresponding results for the conormal version. It is important to keep in mind the following fibre product diagram
\begin{displaymath}
\xymatrix{
\Pbf\Big(p^*(T^*M\oplus f^*T^*N)\Big) \times \Pbf^1 \ar[r] \ar[d]& \Pbf(T^*M\oplus f^*T^*N) \times \Pbf^1 \ar[d]^{\pi \times id} \\
G \times \Pbf^1 \ar[r]^{p \times id} &M \times \Pbf^1
}
\end{displaymath}

Though being trivial itself, it is the key diagram to link the Nash side to the conormal side of the picture.

First, we have 
\begin{displaymath}
\xymatrix{
\Pbf\Big(\pr^*_1S\vert_{\mathscr{M}}\Big) \ar[d] \ar[r] & \Pbf\Big(p^*(T^*M\oplus f^*T^*N)\Big) \times \Pbf^1 \ar[d]\\
\mathscr{M} \ar[r] &G \times \Pbf^1
}
\end{displaymath}
where the horizontal arrows are inclusions. Let $\overline{\mathscr{M}}$ and $\overline{\Pbf\Big(\pr^*_1S\vert_{\mathscr{M}}\Big)}$ be the closures in the spaces on the right column. It is clear that
\begin{equation}\label{bundle}
\overline{\Pbf\Big(\pr^*_1S\vert_{\mathscr{M}}\Big)} = \Pbf\Big(\pr^*_1S\vert_{\overline{\mathscr{M}}}\Big).
\end{equation}

Similarly we consider the following diagram
\begin{equation*}
\begin{tikzcd}[column sep=small]
\Pbf(\Gamma) \arrow[rd] \arrow[rr] & &  \Pbf(T^*M\oplus f^*T^*N)\times \Pbf^1 \arrow[ld] \\
                                            & M\times \Pbf^1 &
\end{tikzcd}
\end{equation*}
and let $\overline{\Pbf(\Gamma)}$ be the closure of $\Pbf(\Gamma)$ in $\Pbf(T^*M\oplus f^*T^*N) \times \Pbf^1$.

It is crucial to note that there is a diagram
\begin{equation*}
\begin{tikzcd}
\Pbf\Big(\pr^*_1S\vert_{\mathscr{M}}\Big) \arrow[r,"\cong"] \arrow[d] & \Pbf(\Gamma) \arrow[d] \\
\mathscr{M} \arrow[r,"\cong"]                                                              & M \times \Cbb
\end{tikzcd}
\end{equation*}
which induces a proper modification
\begin{equation*}
\tilde{p}:\overline{\Pbf\Big(\pr^*_1S\vert_{\mathscr{M}}\Big)} \to \overline{\Pbf(\Gamma)}.
\end{equation*}

The Cartesian squares
\begin{displaymath}
\xymatrix{
\mathscr{M}_{\infty} \ar[r] \ar[d] & \overline{\mathscr{M}} \ar[d] \\
G \times \{\infty\} \ar[r] \ar[d] & G \times \Pbf^1 \ar[d] \\
\{\infty\} \ar[r]& \Pbf^1
}
\end{displaymath}
defines $\mathscr{M}_{\infty}$ as an analytic subspace of $G$. We can define $[\mathscr{M}_{\infty}]$ as the associated cycle of $\mathscr{M}_{\infty}$, or we can define it by $j^![\overline{\mathscr{M}}]$. The agreement of these two definitions is ensured by the formula $\displaystyle{\alpha_t} = \sum_{\mathscr{V}_t\not\subset Y_t}n_i [(V_i)_t]$ on page 176 of \cite{MR732620}.

Let $S_\infty = \pr^*_1S\vert_{\mathscr{M_{\infty}}}$ and let $\Pbf(\Gamma_{\infty})$ be defined by the following Cartesian squares
\begin{equation*}
\xymatrix{
\Pbf(\Gamma_{\infty}) \ar[r] \ar[d] & \overline{\Pbf(\Gamma)} \ar[d] \\
M \times \{\infty\} \ar[r] \ar[d] & M \times \Pbf^1 \ar[d] \\
\{\infty\} \ar[r]^{j} & \Pbf^1 .
}
\end{equation*}

Using the fact that Gysin morphism commutes with flat pull-back and equation \eqref{bundle}, we find that $[\Pbf(S_\infty)]=j^![\overline{\Pbf\Big(\pr^*_1S\vert_{\mathscr{M}}\Big)}]$. Reasoning as we did for the formula $[\mathscr{M}_{\infty}] = j^![\overline{\mathscr{M}}]$, we have $[\Pbf(\Gamma_{\infty})]= j^![\overline{\Pbf(\Gamma)}]$. The commutativity of Gysin pull-back and proper push-forward implies that $\tilde{p}_*[\Pbf(S_\infty)] = [\Pbf(\Gamma_{\infty})]$, and equation \eqref{deflimit} implies that $[\Pbf(S_\infty)] = [\Pbf\Big(\pr^*_1S\vert_{\mathscr{M_{\infty}}}\Big)] = [\Pbf\Big(\pr^*_1S\vert_{\tilde{\mathscr{M}}}\Big)] + [\Pbf\Big(\pr^*_1S\vert_{W}\Big)]$. Combing these equations, we get $[\Pbf(\Gamma_{\infty})] = \tilde{p}_*[\Pbf\Big(\pr^*_1S\vert_{\tilde{\mathscr{M}}}\Big)] + \tilde{p}_*[\Pbf\Big(\pr^*_1S\vert_{W}\Big)]$.

The morphism $f^*T^*N \to T^*M$ is injective when restricted to $M \setminus C(f)$. Set $\Pbf(f^*T^*N)^{\circ} = \Pbf(f^*T^*N)\vert_{M \setminus C(f)}$ and let $\overline{\Pbf(f^*T^*N)^{\circ}}$ be the closure of $\Pbf(f^*T^*N)^{\circ}$ in $\Pbf(T^*M)$. It is clear that $\tilde{p}: \Pbf\Big(\pr^*_1S\vert_{\tilde{\mathscr{M}}}\Big) \to \overline{\Pbf(f^*T^*N)^{\circ}}$ is a proper modification. Let $\Sigma_f$ be the closure of $\Pbf(\Gamma_{\infty})\setminus \overline{\Pbf(f^*T^*N)^{\circ}}$ in $\Pbf(\Gamma_{\infty})$ so that $[\Pbf(\Gamma_{\infty})] = [\overline{\Pbf(f^*T^*N)^{\circ}}] + [\Sigma_f]$. The discussion we have done so far implies the following proposition.

\begin{proposition}
$[\Sigma_f]$ is a projectivised conic Lagrangian cycle of dimension $m+n-1$. The ambient space which contains it as a projectivised Lagrangian cycle is $\Pbf(T^*M \oplus f^*T^*N) \cong \Pbf(T^*(M \times N)\vert_{M})$. Here $M$ is embedded in $M \times N$ via the graph of $f$. We also have 
\begin{align*}
&\tilde{p}_*[\Pbf(S_\infty)] = [\Pbf(\Gamma_{\infty})] = [\overline{\Pbf(f^*T^*N)^{\circ}}] + [\Sigma_f], \\
&\tilde{p}_*[\Pbf\Big(\pr^*_1S\vert_{\tilde{\mathscr{M}}}\Big)] = [\overline{\Pbf(f^*T^*N)^{\circ}}], \\
&\tilde{p}_*[\Pbf\Big(\pr^*_1S\vert_{\widetilde{C(f)}}\Big)] = [\Sigma_f].
\end{align*}

\end{proposition}

\begin{proof}
All the statements are clear except that we still need to show $|\Sigma_f|$ is the projectivisation of a conic Lagrangian subvariety of $T^*(M \times N)$. This statement is local on $M$. Therefore the proof is reduced to the case $N = \mathbb{C}^n$. We have demonstrated in \S \ref{deformation} that the graph deformation can be converted back to the deformation of conormal spaces by using the family of embeddings $i_t$. Though the limits of the two deformations are not the same, there exists a common piece $|\Sigma_f|$ for both. For this, one can safely rely on the geometric intuition provided by Figure \ref{f2} and \ref{f3} since this statement is set-theoretic, i.e. it does not involve the multiplicities of the limit cycles. Since we know that the limit of 1-parameter deformation of Lagrangian family is still Lagrangian, we conclude that $|\Sigma_f|$, as the union of some irreducible components of the limit, must be Lagrangian. The details are left to the reader.
\end{proof}

\begin{remark}
The cycle $[\overline{\Pbf(f^*T^*N)^{\circ}}]$ has dimension $m+n-1$ too, but it is {\em{not}} a projectivised conic Lagrangian cycle, i.e. it is not the projectivised conormal space of any subspace of $M\times N$.  
\end{remark}

Theorem \ref{nash-integral} and Lemma \ref{nash-conor} now helps us to conclude:

\begin{theorem}\label{mu-chi-conor}
Let $f: M \to N$ be flat, sans \'{e}clatement en codimension $0$.
\begin{align*}
\mu(z) &= (-1)^m \int c(\zeta^{\vee}_\infty) \cap s\Big(\pi^{-1}(z)\cap \Sigma_f,\Sigma_f\Big), \\
\chi(z) &= (-1)^{n-1} \int c(\zeta^{\vee}_\infty) \cap s\Big(\pi^{-1}(z)\cap \overline{\Pbf(f^*T^*N)^{\circ}},\overline{\Pbf(f^*T^*N)^{\circ}}\Big) \\
&= (-1)^{n-1} \int c(\zeta^{\vee}_M) \cap s\Big(\pi^{-1}_{M}(z)\cap \overline{\Pbf(f^*T^*N)^{\circ}},\overline{\Pbf(f^*T^*N)^{\circ}}\Big).
\end{align*}
Therefore, the projectivised characteristic cycle of $\mu$ in terms of the graph embedding $C(f) \to M \times N$ is $(-1)^{n-1}[\Sigma_f]$.
\end{theorem}

Though the notation is getting heavier, the geometry is quite simple. $\Sigma_f$ is just the part which does not dominate $M$ in the deformation limit of $\Pbf(\Gamma_t)$. It is mapped into $C(f)$.

\subsection{}\label{application}
Finally, we come to the proof of Theorem \ref{chern-mu}. We need more notations.
\begin{itemize}
\item $\pi_N$ is the projection $\Pbf(f^*T^*N) \to M$.
\item $\xi_N$ is the tautological line bundle of $\Pbf(f^*T^*N)$. Composing $\xi_N \to \pi_N^*f^*T^*N$ and $\pi_N^*f^*T^*N \to \pi_N^*T^*M$, we get a map $\xi_N \to \pi_N^*T^*M$. We can also perform the graph construction to this map. The various graphs ($t\neq \infty$) are denoted by $\gamma_t \subset \xi_N \oplus \pi^*_NT^*M$. Let $a: \Pbf(\xi_N \oplus \pi^*_NT^*M) \to \Pbf(f^*T^*N)$ be the projection. Let $\eta$ and $\tau$ be the tautological subbundle and quotient bundle on $\Pbf(\xi_N \oplus \pi^*_NT^*M)$ respectively.
 \end{itemize}

\begin{proposition}\label{bira}
There is a proper modification $\kappa:\Pbf(\xi_N \oplus \pi^*_NT^*M) \to \Pbf(T^*M \oplus f^*T^*N)$.We also have $\kappa(\Pbf(\gamma_t)) = \Pbf(\Gamma_t)$ and $\kappa^*\xi_N = \eta$.
\end{proposition}

\begin{proof}
A point on $\Pbf(\xi_N \oplus \pi^*_NT^*M)$ consists of the triple
\begin{itemize} 
\item a point $z\in M$, 
\item a 1-dimensional vector space $l_1 \subset f^*T^*N(z)$,
\item another 1-dimensional vector space $l_2 \subset l_1 \oplus T^*M(z)$.
\end{itemize}
The proper modification $\kappa$ is given by $(z,l_1,l_2) \mapsto (z,l_2)$. The other statements are now clear.
\end{proof}

The situation is best summarised in the following diagram.
\begin{equation}\label{lastdiagram}
\begin{tikzcd}
\Pbf(\xi_N \oplus \pi^*_NT^*M) \arrow[r,"\kappa"] \arrow[d,"a"]& \Pbf(T^*M \oplus f^*T^*N) \arrow[d,"\pi"] \\
\Pbf(f^*T^*N) \arrow[r,"\pi_N"] \arrow[ur] & M
\end{tikzcd}
\end{equation}

Now, the relation $\kappa(\Pbf(\gamma_t)) = \Pbf(\Gamma_t)$ implies that $\kappa_*[\Pbf(\gamma_\infty)] = [\Pbf(\Gamma_\infty)]$ by the commutativity of Gysin morphism and proper push-forward. Here we leave the obvious definition of $[\Pbf(\gamma_\infty)]$ to the reader.

We have $[\Pbf(\gamma_\infty)] = [\Pbf(C_Z \oplus \xi_N\vert_{Z})] + [\textup{Bl}_Z\Pbf(f^*T^*N)]$ by \cite{MR732620} Example 18.1.6 (d). Here $Z$ is defined by the zero of the section $\Osr_{\Pbf(f^*T^*N)} \to \pi_N^{*}T^*M \otimes \xi_N^{\vee}$, or equivalently the ideal sheaf of $Z$ is the image of $\pi^*_NTM\otimes \xi_N \to \Osr_{\Pbf(f^*T^*N)}$. The cone $C_Z$ is naturally a subcone of $\pi^*_NT^*M$, and is {\em{not}} the normal cone to $Z$ in $\Pbf(f^*T^*N)$. In fact, they are different by a twist of $\xi_N$. See our discussion of the Rees algebra in \S\ref{misc}. Because the map $\textup{Bl}_Z\Pbf(f^*T^*N) \to \overline{\Pbf(f^*T^*N)^{\circ}}$ is a proper modification, we have $\kappa_*[\textup{Bl}_Z\Pbf(f^*T^*N)] = [\overline{\Pbf(f^*T^*N)^{\circ}}]$. Consequently,
\begin{equation}\label{coversigma}
\kappa_*[\Pbf(C_Z \oplus \xi_N\vert_{Z})] = [\Sigma_f].
\end{equation}

\begin{proposition}
\begin{equation*}
c_*(\mu) =(-1)^m\pi_{N*}a_*\Big(c(a^*\zeta^{\vee}_N)c(\tau^{\vee}) \cap [\Pbf(C_Z \oplus \xi_N\vert_{Z})] \Big)
\end{equation*}
\end{proposition}

\begin{proof}
It follows from Theorem \ref{mu-chi-conor}, proposition \ref{bira}, diagram \eqref{lastdiagram} and equation \eqref{coversigma} that
\begin{align*}
c_*(\mu) &=(-1)^m\pi_{N*}a_*\Big(c(a^*\pi^*_NTM)c(a^*\pi^*_NTN)(c(\eta^{\vee}))^{-1} \cap [\Pbf(C_Z \oplus \xi_N\vert_{Z})] \Big) \\
               &=(-1)^m\pi_{N*}a_*\Big(\frac{c(a^*\pi^*_NTM)c(a^*\xi^{\vee}_N)}{c(\eta^{\vee})}\cdot\frac{c(a^*\pi^*_NTN)}{c(a^*\xi^{\vee}_N)} \cap [\Pbf(C_Z \oplus \xi_N\vert_{Z})]\Big) \\
               &=(-1)^m\pi_{N*}a_*\Big(c(\tau^{\vee})c(a^*\zeta^{\vee}_N) \cap [\Pbf(C_Z \oplus \xi_N\vert_{Z})]\Big).
\end{align*}

\end{proof}

Now, Theorem \ref{chern-mu} becomes a simple corollary of this proposition. Because by the assumption of Theorem \ref{chern-mu}, $Z$ is smooth. Therefore the cone $C_Z$ is the same as $\pi^*_NT^*M\vert_{Z}$, and 
\begin{equation*}
 \Pbf(C_Z \oplus \xi_N\vert_{Z}) = \Pbf((\pi^*_NT^*M \oplus \xi_N)\vert_{Z}).
\end{equation*}
Theorem \ref{chern-mu} then follows from
\begin{equation*}
a_*\Big(c(\tau^{\vee})\cap[\Pbf((\pi^*_NT^*M \oplus \xi_N)\vert_{Z})]\Big) = (-1)^m[Z].
\end{equation*}
See example 3.3.3 in \cite{MR732620}.

\bibliographystyle{alpha}
\bibliography{liao}

\end{document}